\documentclass[12pt]{article}%
\usepackage{rotating}
\usepackage{amsmath}
\usepackage{amsfonts}
\usepackage{amssymb}
\usepackage{graphicx}
\usepackage{placeins}
\usepackage{geometry}
\usepackage{boxedminipage}
\usepackage[usenames]{color}
\usepackage[bookmarks,colorlinks=true,pdfstartview=FitH,citecolor=darkblue,linkcolor=darkblue,urlcolor=darkblue]%
{hyperref}
\usepackage[square,numbers]{natbib}
\usepackage{cleveref}
\usepackage{caption}
\usepackage{amsthm}
\usepackage{longtable}%
\definecolor{darkblue}{rgb}{0.02,0.14,0.3}
\newtheorem{theorem}{Theorem}[section]

\newtheorem{algorithm}{Algorithm}

\crefname{proposition}{Proposition}{Propositions}
\Crefname{proposition}{Proposition}{Propositions}
\crefname{cor}{Corollary}{Corollaries}
\Crefname{cor}{Corollary}{Corollaries}
\crefname{table}{Table}{Tables}
\Crefname{table}{Table}{Tables}
\crefname{figure}{Figure}{Figures}
\Crefname{figure}{Figure}{Figures}
\crefname{section}{Section}{Sections}
\Crefname{section}{Section}{Sections}
\crefname{appendix}{Appendix}{Appendices}
\Crefname{appendix}{Appendix}{Appendices}
\crefname{equation}{equation}{equation}
\Crefname{equation}{Equation}{Equation}
\creflabelformat{equation}{#2\textup{(#1)}#3}
\numberwithin{equation}{section}
\begin{document}

\title{Anderson acceleration with adaptive relaxation for convergent fixed-point iterations}
\author{Nicolas Lepage-Saucier\thanks{Concordia University\newline Department of
Economics\newline Sir George William Campus\newline1455 De Maisonneuve Blvd.
W.\newline Montreal, Quebec, Canada\newline H3G 1M8\medskip\newline
nicolas.lepagesaucier@concordia.ca}}
\date{August 2024}
\maketitle

\begin{abstract}
Two adaptive relaxation strategies are proposed for Anderson acceleration. They are specifically designed for applications in which mappings converge to a fixed point. Their superiority over alternative Anderson acceleration is demonstrated for linear contraction mappings. Both strategies perform well in three nonlinear fixed-point applications that include partial differential equations and the EM algorithm. One strategy surpasses all other Anderson acceleration implementations tested in terms of computation time across various specifications, including composite Anderson acceleration.

\medskip

%


\end{abstract}

\section{Introduction}

Anderson acceleration (AA), also referred to as Anderson mixing, was
formulated by Donald Anderson in 1965 \cite{Anderson1965} to solve integration
problems numerically. Since then, its use has expanded to a wide array of
problems in physics, mathematics and other disciplines. It is very close to
methods developed in other contexts such as Pulay mixing (direct inversion in
the iterative subspace) in chemistry \cite{Pulay1980} or the generalized
minimal residual method (GMRES) and its predecessors like the generalized
conjugate residual method (GCR) in the linear case (see \cite{Saad1986},
\cite{Walker2011} and \cite{Potra2013} for comparisons). There are two
versions of the methods, related to type I and type II Broyden methods. This
paper focuses on the latter.

Consider the mapping $g:%
\mathbb{R}
^{n}\rightarrow%
\mathbb{R}
^{n}$ with at least one fixed point: $g(x^{\ast})=x^{\ast}$. In what follows,
$f\left(  x\right)  =$ $g(x)-x$ can be interpreted as a residual of $x$,
$\left\Vert x\right\Vert \equiv\left\Vert x\right\Vert _{2}=\sqrt
{x^{\intercal}x}$ is the Euclidean$\ $norm of a vector $x$, and $\langle
x,y\rangle=x^{\intercal}y$ is the inner product of the vectors $x$ and $y$.

The AA algorithm can be described as follows.

\begin{algorithm}%
\label{algo:AA}%
Input: a mapping $g:%
\mathbb{R}
^{n}\rightarrow%
\mathbb{R}
^{n}$, a starting point $x_{0}\in%
\mathbb{R}
^{n}$, and an integer $1\leq m\leq n$
\end{algorithm}

\vspace*{-0.5cm}%

\noindent\makebox[\linewidth]{\rule{\textwidth}{0.4pt}}%
%

\renewcommand{\arraystretch}{1.2}%
%

\noindent
\begin{tabular}
[c]{p{0.6cm}l}%
\multicolumn{1}{l}{1} & Set $x_{1}=g(x_{0})$\\
2 & for $k=1,2,...$ until convergence\\
\multicolumn{1}{l}{3} & \qquad Compute $g\left(  x_{k}\right)  $\\
4 & \qquad Compute $(\alpha_{1}^{(k)},...,\alpha_{m_{k}}^{(k)})$ that solve\\
& \qquad\quad$\min_{\alpha_{1}^{(k)},...,\alpha_{m_{k}}^{(k)}}\left\Vert
\sum_{i=1}^{m_{k}}\alpha_{i}^{(k)}f\left(  x_{k-m_{k}+i}\right)  \right\Vert $
s.t. $\sum_{i=1}^{m_{k}}\alpha_{i}^{(k)}=1$\\
\multicolumn{1}{l}{5} & \qquad Compute $\overline{x}_{k}=\sum_{i=1}^{m_{k}%
}\alpha_{i}^{(k)}x_{k-m_{k}+i}\text{ and }\overline{y}_{k}=\sum_{i=1}^{m_{k}%
}\alpha_{i}^{(k)}g\left(  x_{k-m_{k}+i}\right)  $\\
\multicolumn{1}{l}{6} & \qquad Set $x_{k+1}=\overline{x}_{k}+\beta_{k}%
\cdot\left(  \overline{y}_{k}-\overline{x}_{k}\right)  $\\
\multicolumn{1}{l}{7} & end for
\end{tabular}
%

\renewcommand{\arraystretch}{1}%
%

\noindent\makebox[\linewidth]{\rule{\textwidth}{0.4pt}}%

\vspace*{0.1cm}

The number of lags $1\leq m_{k}\leq\min\left(  k,m\right)  $ is usually set as
high as possible to accelerate convergence while providing an acceptable
conditioning of the linear problem for numerical stability.

The scalar $\beta_{k}$ is the relaxation (damping, mixing) parameter. In most
AA applications, it is usually stationary and is often set to 1. As discussed
by Anderson in \cite{Anderson2019}, the choice of $\beta_{k}=1$ is natural
under the implicit assumption that $g$ converges to a fixed point $x^{\ast}$.
In this case, $g(x_{k})$ should be closer to $x^{\ast}$ than $x_{k}$ and,
correspondingly, $\overline{y}_{k}$ should be closer to $x^{\ast}$ than
$\overline{x}_{k}$, provided that $\overline{y}_{k}$ is a good approximation
of $g(\overline{x}_{k})$. This assumption is of course not always valid.
Smaller $\beta_{k}$ have sometimes been shown to help stability and improve
global convergence for challenging applications (see, for instance,
\cite{Pollock2021} and \cite{Warnock2021}). But for fixed-point iterations
which do converge to a fixed point, this property can be exploited to compute
adaptive relaxation parameters which may significantly speed up the
convergence of AA.

The idea of AA with\ dynamically adjusted relaxation is a relatively recent
one. In 2019, Anderson \cite{Anderson2019} outlined an algorithm for adjusting
$\beta_{k}$ by monitoring the convergence of the Picard iteration for $g(x)$
over multiple iterations. Focusing on linearly-converging fixed-point methods,
Evans et al. \cite{Evans2020} used the optimization gain $\theta_{k+1}$ from
an AA step to set $\beta_{k}=0.9-\theta_{k+1}/2$. They show that this choice
of relaxation parameter improved robustness and efficiency over constant
relaxation with $\beta=0.4$ and $\beta=.75$ in various applications. Recently,
Jin et al. \cite{Jin2024} incorporated a dynamically adjusted $\beta_{k}$ in
their Anderson acceleration of the derivative-free projection method.

In 2013, Potra and Engler \cite{Potra2013} extended the results of
\cite{Walker2011} by characterizing the behavior of AA for linear systems with
arbitrary non-zero relaxation parameters. While doing so, they also suggested
a simple procedure to compute an optimal $\beta_{k}$ minimizing the residual
of $x_{k+1}$. Chen and Vuik \cite{Chen2024} extended their method to nonlinear
problems with a locally optimal $\beta_{k}$ that minimizes $\left\Vert
f\left(  x_{k+1}\right)  \right\Vert $. Using the approximation $g\left(
\overline{x}_{k}+\beta_{k}\left(  \overline{y}_{k}-\overline{x}_{k}\right)
\right)  \approx g\left(  \overline{x}_{k}\right)  +\beta_{k}\left(  g\left(
\overline{y}_{k}\right)  -g\left(  \overline{x}_{k}\right)  \right)  $, this
optimal $\beta_{k}$ is%
\begin{equation}
\beta_{k}^{\ast}=\arg\min_{\beta}\left\Vert f\left(  \overline{x}_{k}\right)
+\beta\cdot\left(  f\left(  \overline{y}_{k}\right)  -f\left(  \overline
{x}_{k}\right)  \right)  \right\Vert =-\frac{\left\langle f\left(
\overline{y}_{k}\right)  -f\left(  \overline{x}_{k}\right)  ,f\left(
\overline{x}_{k}\right)  \right\rangle }{\left\Vert f\left(  \overline{y}%
_{k}\right)  -f\left(  \overline{x}_{k}\right)  \right\Vert ^{2}}
\label{eq:opt_beta}%
\end{equation}
and the optimal AA update is
\begin{equation}
x_{k+1}^{0}=\overline{x}_{k}+\beta_{k}^{\ast}\left(  \overline{y}%
_{k}-\overline{x}_{k}\right)  . \label{eq:opt_update}%
\end{equation}

This scheme will be noted AAopt0. In various numerical applications, Chen and
Vuik show that AAopt0 needs fewer iterations to converge compared to regular
AA. Since AAopt0 needs two extra maps per iteration, $g\left(  \overline
{x}_{k}\right)  $ and $g\left(  \overline{y}_{k}\right)  $, they suggest using
their method for applications for which mappings are inexpensive to compute.
Otherwise, the potential efficiency gains from a smaller number of iterations
may be offset by the extra computation per iteration.

\subsection{Locally optimal relaxation for convergent mapping applications}

Two modifications to AAopt0 are now proposed to improve its performance for
convergent mapping iterations. First, note that (\ref{eq:opt_beta}) and
(\ref{eq:opt_update}) constitute by themselves a mini order-1 AA iteration.
Using $\overline{x}_{k}$ and $\overline{y}_{k}$ to compute $x_{k+1}^{0}$ means
implicitly choosing full relaxation ($\beta=0$) for this mini AA iteration.
But for convergent mappings, we can hope that $g\left(  \overline{x}%
_{k}\right)  $ and $g\left(  \overline{y}_{k}\right)  $ are closer to
$x^{\ast}$ than $\overline{x}_{k}$ and $\overline{y}_{k}$, respectively. To
take full advantage of the work of computing them, a more natural choice is
thus the non-relaxed mini AA%
\begin{equation}
x_{k+1}^{1}=g\left(  \overline{x}_{k}\right)  +\beta_{k}^{\ast}\left(
g\left(  \overline{y}_{k}\right)  -g\left(  \overline{x}_{k}\right)  \right)
. \label{eq:opt_update1}%
\end{equation}
This modified scheme will be labeled AAopt1. Section \ref{sec:proof_AAopt1}
presents a proof that AAopt1 improves convergence compared to AAopt0 for
contractive linear mappings.

Second, the empirical tests shown below suggest that the optimal relaxation
parameters $\beta_{k}^{\ast}$ tend to be correlated between iterations. This
opens up the possibility of the same $\beta_{k}^{\ast}$ as approximation for
$\beta_{k+1}^{\ast},\beta_{k+2}^{\ast},...$, as long as the lower precision
due to the approximation does not deteriorate convergence too much. An AAopt1
which only updates the relaxation every $T$ iterations will be noted
AAopt1\_T, with AAopt1\_1 $\equiv$ AAopt1. Of course, the same modification
could be applied to AAopt0. Details of the algorithm are presented in Section
\ref{sec:AAopt_T_algo}.

\subsection{Costless adaptive relaxation for convergent
mappings\label{sec:Costless}}

A new adaptive relaxation parameter for AA is now proposed. It requires only
two inner products and no extra maps per iteration.

Assume that $x_{k+1}=\overline{x}_{k}+\beta_{k}\left(  \overline{y}%
_{k}-\overline{x}_{k}\right)  $ has been computed for some $\beta_{k}$. Since
the mapping $g$ is assumed to converge to $x^{\ast}$, the next map used in the
AA algorithm, $g\left(  x_{k+1}\right)  $, should provide useful information
on the location of $x^{\ast}$. An improved choice of relaxation parameter in
iteration $k$ would be the $\hat{\beta}_{k}$ that minimizes the distance
between $g\left(  x_{k+1}\right)  $ and $\overline{x}_{k}+\hat{\beta}%
_{k}\left(  \overline{y}-\overline{x}_{k}\right)  $:%
\begin{equation}
\hat{\beta}_{k}=\arg\min_{\beta}\left\Vert \overline{x}_{k}+\beta\left(
\overline{y}_{k}-\overline{x}_{k}\right)  -g\left(  x_{k+1}\right)
\right\Vert _{2}=\frac{\left\langle \overline{y}_{k}-\overline{x}_{k},g\left(
x_{k+1}\right)  -\overline{x}_{k}\right\rangle }{\left\Vert \overline{y}%
_{k}-\overline{x}_{k}\right\Vert ^{2}}. \label{eq:beta_hat}%
\end{equation}

Section \ref{sec:AAnm_proof} includes a demonstration that $\hat{\beta}_{k}$
can improve convergence for a contractive linear operators under an
appropriate norm.

In practice, since $g\left(  x_{k+1}\right)  $ is needed to compute
$\hat{\beta}_{k}$, using it at iteration $k$ to recompute an improved
$x_{k+1}$ implies computing an extra map. However, as stated before, if
$\hat{\beta}_{k}$ is correlated with $\hat{\beta}_{k+1}$, $\hat{\beta}_{k}$ it
can be used as an approximation for $\hat{\beta}_{k+1}$. Computing $\hat
{\beta}_{k}$ would then require only two inner products, a negligible cost
within the AA algorithm.

This scheme, which minimizes the distance between $x_{k+1}$ and its map, will
be labeled AAmd. Its implementation details are provided in Section
\ref{sec:AAnm}. As will be seen in Section \ref{sec:applications}, AAmd
improves convergence speeds for all nonlinear applications presented.

\subsection{Regularizing $\beta\label{sec:Regularizing}$}

Essentially all the AA literature assumes $0<\beta_{k}\leq1$. One exception
could in some way be SqS1-SQUAREM\ \cite{Varadhan2008}, as discussed in
\cite{Tang2023}. It corresponds to a 1-iteration AA where $\alpha$ is also
used as the relaxation parameter. Varadhan and Roland show that this choice
can greatly speed up convergence. The same was pointed out by Raydan and
Svaiter \cite{Raydan2002} for a linear version of the algorithm. Since the
resulting values for $\beta_{k}$ depend on the eigenvalues of the Jacobian of
$g(x)$, these can take values well above one.

A striking result from the numerical simulations presented in Section
\ref{sec:Examples} is that for both suggested schemes, and especially for
AAmd, the calculated relaxation parameters are regularly above one. When this
happens, Chen and Vuik \cite{Chen2024} recommend falling back to a default
relaxation parameter below 1. For AAopt1\_T and AAmd, empirical tests suggest
that relaxation parameters above 1 can provide good results, as long as
excessively high values are avoided. To do so, a practical option is simply to
cap $\beta_{k}$ at some maximum value $\beta_{\max}>1$.\ For all the numerical
tests presented below, $\beta_{\max}=3$ was used.

\subsection{Testing framework}

The performances of AAopt1\_T and AAmd will be tested in three fixed-point
application: a Poisson partial differential equation (PDE) and two
expectation-maximization (EM) algorithm \cite{Dempster1977}. Since AAopt1\_T
and AAmd vary in terms of mappings per iteration, the number of iterations or
maps until convergence are not the most informative criteria for comparison.
For practitioners and software library designers, the important benchmark is
ultimately computation speed. In practice however, comparing computation times
meaningfully is challenging since they can be affected by software
implementation, hardware choices, and programming languages. To get the most
complete picture, iterations, mappings, and computation times will be shown.
To make the speed comparisons as informative as possible, all problems were
programmed as efficiently as possible in a compiled programming language and
executed on the same hardware.

Another matter of concern when comparing multiple AA implementations is the
choice of parameters, especially the number of lags $m$. With a sufficiently
large $m$, AA with constant relaxation can often converge in a similar number
of iterations as AAopt0 or AAopt1. Of course, a larger $m$ entails more work
solving the linear optimization problem and potentially less precision due to
worse conditioning. To properly compare the various versions of AA, each of
them should be implemented with an optimal $m$ for each problem, i.e., the one
which minimizes computation times.

Finally, initial starting points can play a major and unpredictable role in
convergence rates, especially for non-linear applications. To make the
comparisons more robust, numerical experiments with many different starting
points and different synthetic data will be performed.

The applications details, the test implementations and the results are
presented in Section \ref{sec:applications}. Section \ref{sec:Conclusion}
offers concluding remarks.

\section{AAopt1\_T$\label{sec:AAopt1_T}$}

This section demonstrates the convergence improvement of AAopt1 compared to
AAopt0 for a contractive linear operator. This advantage is then illustrated
by solving numerically a linear system of equations. Finally, the full
AAopt1\_T algorithm is laid out in detail.

\subsection{Local improvements from AAopt1 for a linear contraction
mapping\label{sec:proof_AAopt1}}

Consider solving $Ax=b$ where $x\in%
\mathbb{R}
^{n}$, $b\in%
\mathbb{R}
^{n}$ and $A\in%
\mathbb{R}
^{n\times n}$ is symmetric with eigenvalues $0<\lambda_{i}<2$ for all $i$. To
solve this system of equations, define the linear mapping $g(x)=x-\left(
Ax-b\right)  $.

At iteration $k$ of the AA algorithm, assume that $\overline{y}_{k}$ and
$\overline{x}_{k}$ (as in step 5 of Algorithm \ref{algo:AA}) have been computed.

\begin{theorem}%
\label{theorem_AAopt}%
Let $f(x)=g(x)-x=-\left(  Ax-b\right)  $ with $b\in%
\mathbb{R}
^{n}$ and $A\in%
\mathbb{R}
^{n\times n}$ symmetric with eigenvalues $0<\lambda_{i}<2$ for $i=1,...,n$.
For any real vectors $\bar{x}_{k},\bar{y}_{k},$%
\[
\left\Vert f\left(  x_{k+1}^{1}\right)  \right\Vert <\left\Vert f\left(
x_{k+1}^{0}\right)  \right\Vert ,
\]
as long as $f\left(  x_{k+1}^{0}\right)  \neq\mathbf{0}$ (AA has not
converged), where $x_{k+1}^{0}$ and $x_{k+1}^{1}$ are defined in
(\ref{eq:opt_update}) and (\ref{eq:opt_update1}), respectively.
\end{theorem}

\begin{proof}
It is straightforward to rewrite $x_{k+1}^{1}$ in terms of $x_{k+1}^{0}$:%
\begin{align*}
x_{k+1}^{1}  &  =\left(  1-\beta_{k}^{\ast}\right)  g\left(  \overline{x}%
_{k}\right)  +\beta_{k}^{\ast}g\left(  \overline{y}_{k}\right) \\
x_{k+1}^{1}  &  =\left(  1-\beta_{k}^{\ast}\right)  \left[  \overline{x}%
_{k}-\left(  A\overline{x}_{k}-b\right)  \right]  +\beta_{k}^{\ast}\left[
\overline{y}_{k}-\left(  A\overline{y}_{k}-b\right)  \right] \\
x_{k+1}^{1}  &  =x_{k+1}^{0}-\left(  Ax_{k+1}^{0}-b\right) \\
x_{k+1}^{1}  &  =x_{k+1}^{0}+f\left(  x_{k+1}^{0}\right)  .
\end{align*}

Expressing $f\left(  x_{k+1}^{1}\right)  $ in terms of $f\left(  x_{k+1}%
^{0}\right)  $:%
\begin{align*}
f\left(  x_{k+1}^{1}\right)   &  =-\left(  A\left(  x_{k+1}^{0}+f\left(
x_{k+1}^{0}\right)  \right)  -b\right) \\
f\left(  x_{k+1}^{1}\right)   &  =-\left(  Ax_{k+1}^{0}-b\right)  -Af\left(
x_{k+1}^{0}\right) \\
f\left(  x_{k+1}^{1}\right)   &  =\left(  I-A\right)  f\left(  x_{k+1}%
^{0}\right)  .
\end{align*}
The squared 2-norms of $f\left(  x_{k+1}^{0}\right)  $ and $f\left(
x_{k+1}^{1}\right)  $ are
\[
\left\Vert f\left(  x_{k+1}^{0}\right)  \right\Vert ^{2}=f\left(  x_{k+1}%
^{0}\right)  ^{\intercal}f\left(  x_{k+1}^{0}\right)
\]
and%
\begin{align*}
\left\Vert f\left(  x_{k+1}^{1}\right)  \right\Vert ^{2}  &  =\left(  \left(
I-A\right)  f\left(  x_{k+1}^{0}\right)  \right)  ^{\intercal}\left(
I-A\right)  f\left(  x_{k+1}^{0}\right) \\
&  =f\left(  x_{k+1}^{0}\right)  ^{\intercal}\left(  I-A\right)  ^{\intercal
}\left(  I-A\right)  f\left(  x_{k+1}^{0}\right)  .
\end{align*}

Since $A$ has eigenvalues $0<\lambda_{i}<2$ for $i=1,...,n$, the product
$\left(  I-A\right)  ^{\intercal}\left(  I-A\right)  $ is symmetric with
eigenvalues $0\leq a_{i}<1$ for $i=1,...,n$. The ratio%
\[
\frac{\left\Vert f\left(  x_{k+1}^{1}\right)  \right\Vert ^{2}}{\left\Vert
f\left(  x_{k+1}^{0}\right)  \right\Vert ^{2}}=\frac{f\left(  x_{k+1}%
^{0}\right)  ^{\intercal}\left(  I-A\right)  ^{\intercal}\left(  I-A\right)
f\left(  x_{k+1}^{0}\right)  }{f\left(  x_{k+1}^{0}\right)  ^{\intercal
}f\left(  x_{k+1}^{0}\right)  }%
\]
is a Rayleigh quotient that can only take values in $[0,1)$. Therefore,
$\frac{\left\Vert f\left(  x_{k+1}^{1}\right)  \right\Vert }{\left\Vert
f\left(  x_{k+1}^{0}\right)  \right\Vert }<1$, which completes the proof.
\end{proof}

\subsection{A linear example with AAopt1\label{sec:linear_example_AAopt}}

To apply the proof in a concrete example, consider using the mapping
$g(x)=x-\left(  Ax-b\right)  $ to find the solution to $Ax=b$ where $A=$
diag$(0.1,0.2,...,1.9)$,$\ b=\mathbf{1}\ $and the initial guess is
$x_{0}=\mathbf{0}$. To highlight the difference between AAopt0 and AAopt1, no
bounds on $\beta_{k}^{\ast}$ is imposed on either of them. AA with constant
$\beta=1$ is also shown for comparison. All are implemented with $m=8$.%

\begin{figure}[tbp]%
\begin{minipage}{0.48\textwidth}%
%

\raisebox{-0cm}{\includegraphics[
height=7.0973cm,
width=7.0973cm
]%
{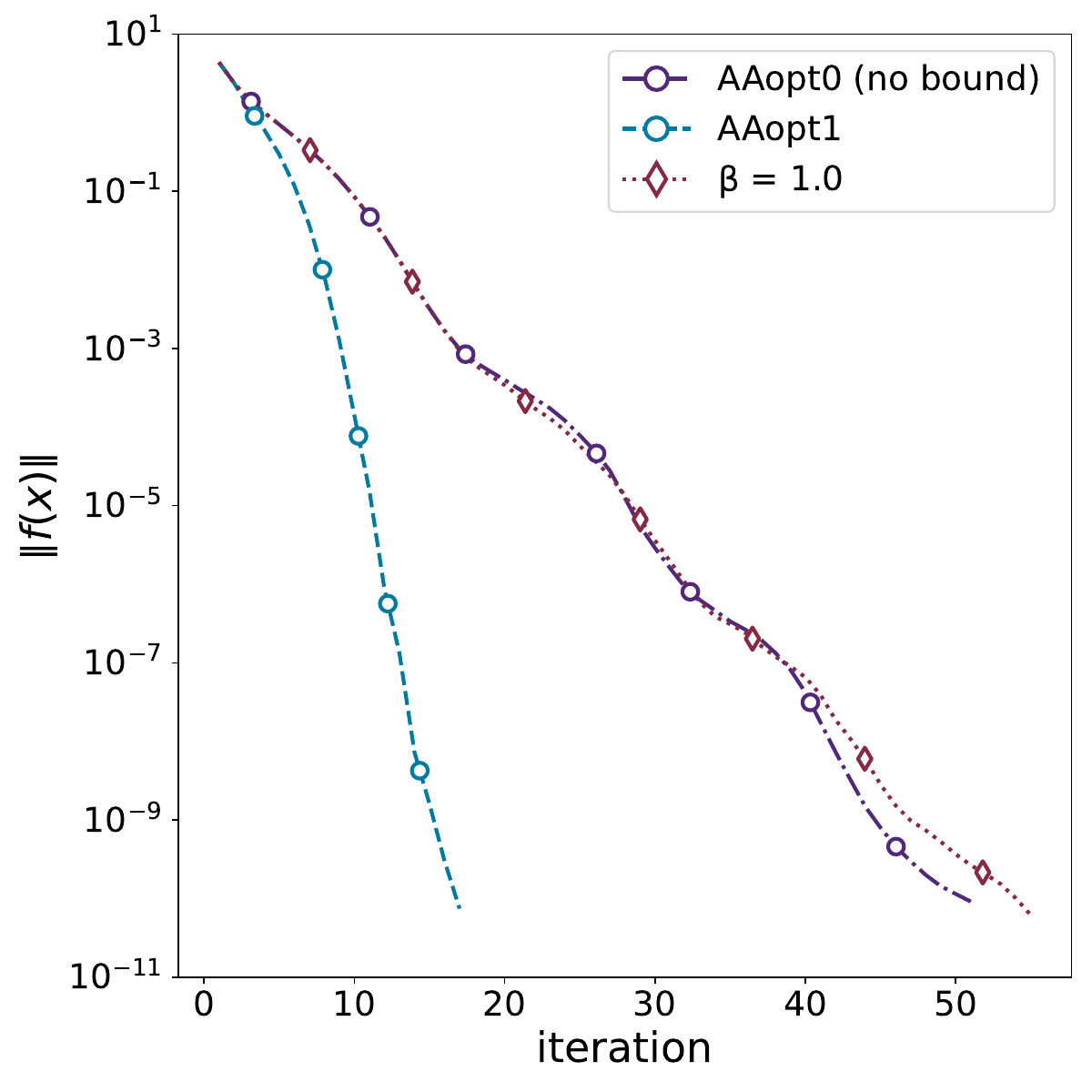}%
}
%

\caption{Residuals (with norm $A^{-1}$), Linear system of equations}%
\label{fig:linear1_AAopt}%
\end{minipage}%
\hfill%
\begin{minipage}{0.48\textwidth}
\raisebox{-0cm}{\includegraphics[
height=7.0973cm,
width=7.0973cm
]%
{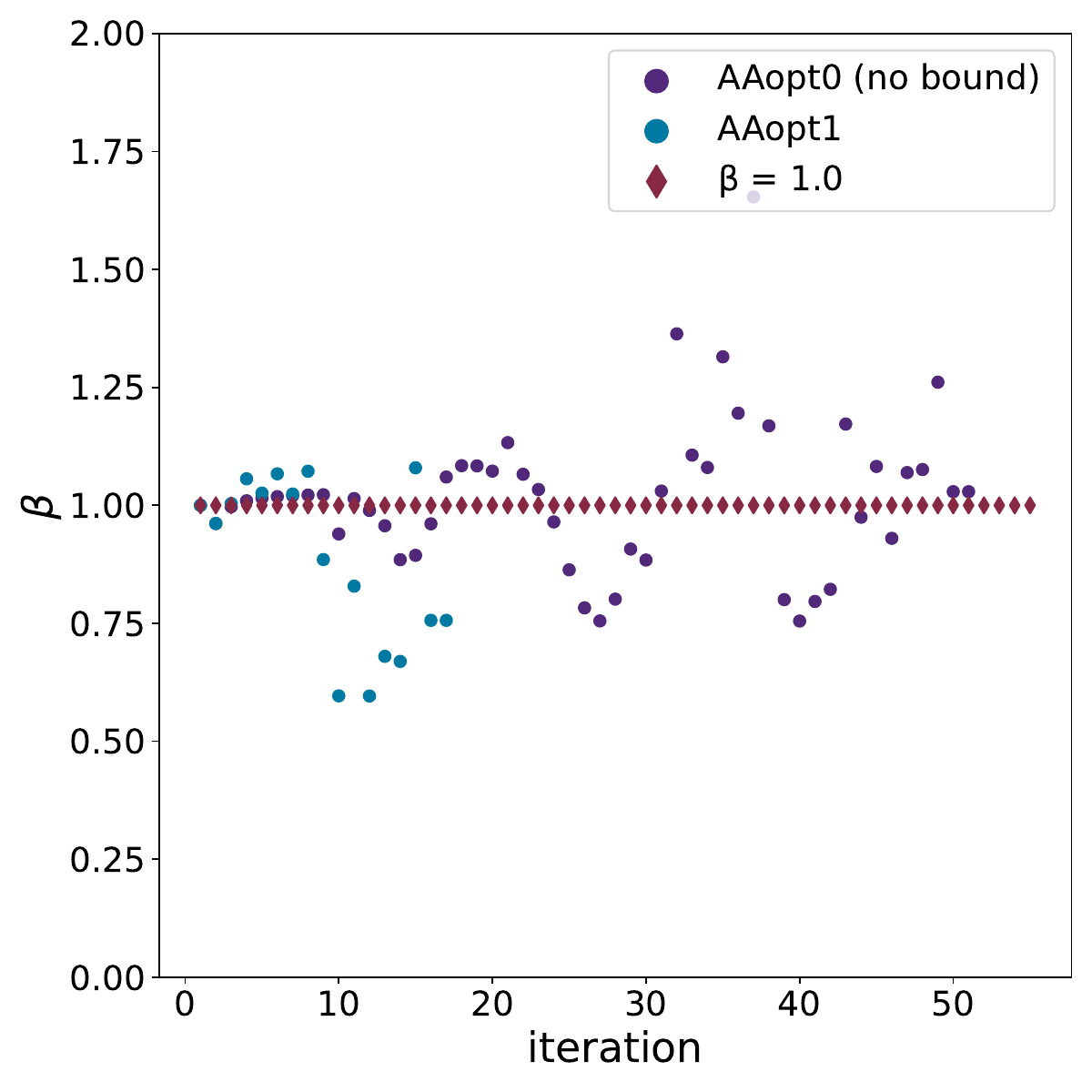}%
}
\caption{Relaxation parameter, linear system of equations}%
\label{fig:linear2_AAopt}%
\end{minipage}%
\end{figure}%

Figure \ref{fig:linear1_AAopt} shows the Euclidean norm of the residual for
each algorithm and Figure \ref{fig:linear2_AAopt} shows the relaxation
parameter at each iteration. AAopt1 converges almost 3 times faster than
AAopt0 and AA with $\beta=1$. This small number of iterations for AAopt1's
looks promising. But since each iteration requires three maps per iteration
instead of one for stationary AA, comparing computation speeds will be a true
test of its promises as an optimization algorithm.

\subsection{The full AAopt1\_T algorithm\label{sec:AAopt_T_algo}}

The full AAopt1\_T algorithm takes as extra parameters $T\geq1$, which
determines at which interval the optimal relaxation $\beta_{k}^{\ast}$ is
recomputed, $\beta_{\max}$, and $\beta_{\text{default}}$ as described in
\ref{sec:Regularizing}.

\begin{algorithm}%
\label{algo:AAopt1_T}%
Input: a mapping $g:%
\mathbb{R}
^{n}\rightarrow%
\mathbb{R}
^{n}$, a starting point $x_{0}\in%
\mathbb{R}
^{n}$, $1\leq m\leq n$, $\beta_{\max}>0,0<\beta_{\text{default}}\leq
\beta_{\max}$ and $T\geq1.$
\end{algorithm}

%

\noindent\makebox[\linewidth]{\rule{\textwidth}{0.4pt}}%
%

\renewcommand{\arraystretch}{1.2}%
%

\noindent
\begin{tabular}
[c]{p{0.6cm}l}%
1 & Set $x_{1}=g(x_{0})$%
\end{tabular}
%

\noindent
\begin{tabular}
[c]{p{0.6cm}l}%
2 & for $k=1,2,...$ until convergence
\end{tabular}
%

\noindent
\begin{tabular}
[c]{p{0.6cm}l}%
3 & \qquad Compute $g\left(  x_{k}\right)  $%
\end{tabular}
%

\noindent
\begin{tabular}
[c]{p{0.6cm}l}%
4 & \qquad Compute $(\alpha_{1}^{(k)},...,\alpha_{m_{k}}^{(k)})$ that solve
\end{tabular}
%

\noindent
\begin{tabular}
[c]{p{0.6cm}l}
& \qquad\quad$\min_{\alpha_{1}^{(k)},...,\alpha_{m_{k}}^{(k)}}\left\Vert
\sum_{i=1}^{m_{k}}\alpha_{i}^{(k)}f\left(  x_{k-m_{k}+i}\right)  \right\Vert $
s.t. $\sum_{i=1}^{m_{k}}\alpha_{i}^{(k)}=1$%
\end{tabular}
%

\noindent
\begin{tabular}
[c]{p{0.6cm}l}%
5 & \qquad Compute $\overline{x}_{k}=\sum_{i=1}^{m_{k}}\alpha_{i}%
^{(k)}x_{k-m_{k}+i}\text{ and }\overline{y}_{k}=\sum_{i=1}^{m_{k}}\alpha
_{i}^{(k)}g\left(  x_{k-m_{k}+i}\right)  $%
\end{tabular}
%

\noindent
\begin{tabular}
[c]{p{0.6cm}l}%
6 & \qquad If $k=1$ or $k\operatorname{mod}T=0$%
\end{tabular}
%

\noindent
\begin{tabular}
[c]{p{0.6cm}l}%
7 & \qquad\qquad Compute $g\left(  \overline{x}_{k}\right)  $ and
$g(\overline{y}_{k})$%
\end{tabular}
%

\noindent
\begin{tabular}
[c]{p{0.6cm}l}%
8 & \qquad\qquad Compute $\beta_{k}^{\ast}=-\frac{\left\langle f\left(
\overline{y}_{k}\right)  -f\left(  \overline{x}_{k}\right)  ,f\left(
\overline{x}_{k}\right)  \right\rangle }{\left\Vert f\left(  \overline{y}%
_{k}\right)  -f\left(  \overline{x}_{k}\right)  \right\Vert ^{2}}$%
\end{tabular}
%

\noindent
\begin{tabular}
[c]{p{0.6cm}l}%
9 & \qquad\qquad If $\beta_{k}^{\ast}\leq0$%
\end{tabular}
%

\noindent
\begin{tabular}
[c]{p{0.6cm}l}%
10 & \qquad\qquad\qquad Set $\beta_{k}=\beta_{\text{default}}$%
\end{tabular}
%

\noindent
\begin{tabular}
[c]{p{0.6cm}l}%
11 & \qquad\qquad else
\end{tabular}
%

\noindent
\begin{tabular}
[c]{p{0.6cm}l}%
12 & \qquad\qquad\qquad Set $\beta_{k}=\min(\beta_{k}^{\ast},\beta_{\max})$%
\end{tabular}
%

\noindent
\begin{tabular}
[c]{p{0.6cm}l}%
13 & \qquad\qquad end if
\end{tabular}
%

\noindent
\begin{tabular}
[c]{p{0.6cm}l}%
14 & \qquad\qquad Set $x_{k+1}=g\left(  \overline{x}_{k}\right)  +\beta
_{k}(g(\overline{y}_{k})-g\left(  \overline{x}_{k}\right)  )$%
\end{tabular}
%

\noindent
\begin{tabular}
[c]{p{0.6cm}l}%
15 & \qquad else
\end{tabular}
%

\noindent
\begin{tabular}
[c]{p{0.6cm}l}%
16 & \qquad\qquad Set $\beta_{k}=\beta_{k-1}$%
\end{tabular}
%

\noindent
\begin{tabular}
[c]{p{0.6cm}l}%
17 & \qquad\qquad Set $x_{k+1}=\overline{x}_{k}+\beta_{k}\left(  \overline
{y}_{k}-\overline{x}_{k}\right)  $%
\end{tabular}
%

\noindent
\begin{tabular}
[c]{p{0.6cm}l}%
18 & \qquad end if
\end{tabular}
%

\noindent
\begin{tabular}
[c]{p{0.6cm}l}%
19 & end for
\end{tabular}
%

\renewcommand{\arraystretch}{1}%
%

\noindent\makebox[\linewidth]{\rule{\textwidth}{0.4pt}}%

Note that possible adjustments to $m_{k}$ such as restarts and composite AA
will be described in Section \ref{sec:Restarts_and_composite_AA}. Also,
$\beta_{\text{default}}=1$, $\beta_{\max}=3$ in all experiments.

\section{AAmd\label{sec:AAnm}}

This section justifies the use of $\hat{\beta}$ as a relaxation parameter and
shows its properties in the same linear example as before. Then, it details
the implementation of the AAmd algorithm.

\subsection{Local improvements from AAmd in a linear contraction mapping\label{sec:AAnm_proof}}

Consider the same mapping as the one for Proof \ref{theorem_AAopt}. Contrary
to AAopt1, AAmd may improve convergence at each iteration in the familiar
Euclidean norm. Nevertheless, we will show that it does in the elliptic norm%
\[
\left\Vert x\right\Vert _{A^{-1}}=\sqrt{x^{\intercal}A^{-1}x}%
\]
induced by the inner product $\left\langle .,.\right\rangle _{A^{-1}}$%

\[
\left\langle x,y\right\rangle _{A^{-1}}=x^{\intercal}A^{-1}y.
\]

Assume that at iteration $k$ of AA, $\overline{y}_{k}$ and $\overline{x}_{k}$
(as in step 5 of Algorithm \ref{algo:AA}) have been computed. To lighten the
notation, define the distance $\bar{d}_{k}\equiv\overline{y}_{k}-\overline
{x}_{k}$. For a given $\beta_{k}$, the next iteration of AA is $x_{k+1}%
=\overline{x}_{k}+\bar{d}_{k}\beta_{k}$ and the map of $x_{k+1}$ is
$g(x_{k+1})$. Using this information, we can compute the improved relaxation
parameter $\hat{\beta}_{k}=\frac{\left\langle \bar{d}_{k},g\left(
x_{k+1}\right)  -\overline{x}_{k}\right\rangle }{\left\Vert \bar{d}%
_{k}\right\Vert ^{2}}$ as defined in (\ref{eq:beta_hat}). Using $\hat{\beta
}_{k}$, we can compute an improved next iterate $\hat{x}_{k+1}=\overline
{x}_{k}+\bar{d}_{k}\hat{\beta}_{k}$. The following theorem proves that
$\hat{\beta}_{k}$ would have been at least as good a choice as $\beta_{k}$ or
better for any $\beta_{k}$, in the sense that $\left\Vert f\left(  \hat
{x}_{k+1}\right)  \right\Vert _{A^{-1}}\leq\left\Vert f\left(  x_{k+1}\right)
\right\Vert _{A^{-1}}$ for all $\beta_{k}$.

\begin{theorem}%
\label{theorem_AAnm}%
Let $f(x)=-\left(  Ax-b\right)  $ with $b\in%
\mathbb{R}
^{n}$ and $A\in%
\mathbb{R}
^{n\times n}$ symmetric with eigenvalues $0<\lambda_{i}<2$ for $i=1,...,n$.
For any $\bar{x}_{k},\bar{d}_{k},\beta_{k},$%
\[
\left\Vert f\left(  \hat{x}_{k+1}\right)  \right\Vert _{A^{-1}}\leq\left\Vert
f\left(  x_{k+1}\right)  \right\Vert _{A^{-1}},
\]
where $x_{k+1}=\overline{x}_{k}+\bar{d}_{k}\beta_{k}$, $\hat{x}_{k+1}%
=\overline{x}_{k}+\bar{d}_{k}\hat{\beta}_{k}$, and $\hat{\beta}_{k}%
=\frac{\left\langle \bar{d}_{k},g\left(  x_{k+1}\right)  -\overline{x}%
_{k}\right\rangle }{\left\Vert \bar{d}_{k}\right\Vert ^{2}}$.
\end{theorem}

\begin{proof}
Substituting $\overline{x}_{k}=x_{k+1}-\bar{d}_{k}\beta_{k}$ in the definition
of $\hat{\beta}_{k}$:%
\begin{align*}
\hat{\beta}_{k}  &  =\frac{\left\langle \bar{d}_{k},g\left(  x_{k+1}\right)
-x_{k+1}+\bar{d}_{k}\beta_{k}\right\rangle }{\left\Vert \bar{d}_{k}\right\Vert
^{2}}\\
\hat{\beta}_{k}  &  =\frac{\left\langle \bar{d}_{k},f\left(  x_{k+1}\right)
\right\rangle }{\left\Vert \bar{d}_{k}\right\Vert ^{2}}+\beta_{k}.
\end{align*}
Note that $\frac{\left\langle \bar{d}_{k},f\left(  x_{k+1}\right)
\right\rangle }{\left\Vert \bar{d}_{k}\right\Vert ^{2}}=0$ -- i.e., $f\left(
x_{k+1}\right)  $ is orthogonal to $\bar{d}_{k}$ -- implies $\hat{\beta}%
_{k}=\beta_{k}$. In other words, $\hat{\beta}_{k}$ is not an improvement over
$\beta_{k}$.

The next iterate $\hat{x}_{k+1}$ is%
\begin{align*}
\hat{x}_{k+1}  &  =\overline{x}_{k}+\bar{d}_{k}\hat{\beta}_{k}\\
\hat{x}_{k+1}  &  =x_{k+1}-\bar{d}_{k}\beta_{k}+\bar{d}_{k}\left(
\frac{\left\langle \bar{d}_{k},f\left(  x_{k+1}\right)  \right\rangle
}{\left\Vert \bar{d}_{k}\right\Vert ^{2}}+\beta_{k}\right) \\
\hat{x}_{k+1}  &  =x_{k+1}+\bar{d}_{k}\frac{\left\langle \bar{d}_{k},f\left(
x_{k+1}\right)  \right\rangle }{\left\Vert \bar{d}_{k}\right\Vert ^{2}}\\
\hat{x}_{k+1}  &  =x_{k+1}-\bar{d}_{k}\left(  \bar{d}_{k}^{\intercal}\bar
{d}_{k}\right)  ^{-1}\bar{d}_{k}^{\intercal}\left(  Ax_{k+1}-b\right)  .\\
\hat{x}_{k+1}  &  =x_{k+1}-P_{\bar{d}_{k}}\left(  Ax_{k+1}-b\right)  .
\end{align*}
where $P_{\bar{d}_{k}}=\bar{d}_{k}\left(  \bar{d}_{k}^{\intercal}\bar{d}%
_{k}\right)  ^{-1}\bar{d}_{k}^{\intercal}$ is a projection matrix. Computing
$f\left(  \hat{x}_{k+1}\right)  $ and expressing it as function of $f\left(
x_{k+1}\right)  $:%
\begin{align*}
f\left(  \hat{x}_{k+1}\right)   &  =-A\left(  x_{k+1}-P_{\bar{d}_{k}}\left(
Ax_{k+1}-b\right)  \right)  +b\\
f\left(  \hat{x}_{k+1}\right)   &  =-\left(  Ax_{k+1}-AP_{\bar{d}_{k}}\left(
Ax_{k+1}-b\right)  -b\right) \\
f\left(  \hat{x}_{k+1}\right)   &  =-\left(  I-AP_{\bar{d}_{k}}\right)
\left(  Ax_{k+1}-b\right) \\
f\left(  \hat{x}_{k+1}\right)   &  =\left(  I-AP_{\bar{d}_{k}}\right)
f\left(  x_{k+1}\right)  .
\end{align*}

If $f\left(  x_{k+1}\right)  =\mathbf{0}$ (the AA algorithm has converged),
$f\left(  \hat{x}_{k+1}\right)  =\mathbf{0}$ and the theorem trivially
verified. If not, we may compute $\left\Vert f\left(  \hat{x}_{k+1}\right)
\right\Vert _{A^{-1}}^{2}$ and $\left\Vert f\left(  x_{k+1}\right)
\right\Vert _{A^{-1}}^{2}$:%
\[
\left\Vert f\left(  x_{k+1}\right)  \right\Vert _{A^{-1}}^{2}=f\left(
x_{k+1}\right)  ^{\intercal}A^{-1}f\left(  x_{k+1}\right)  =\left(
A^{-0.5}f\left(  x_{k+1}\right)  \right)  ^{\intercal}\left(  A^{-0.5}f\left(
x_{k+1}\right)  \right)  .
\]
Similarly,%
\begin{align*}
\left\Vert f\left(  \hat{x}_{k+1}\right)  \right\Vert _{A^{-1}}^{2}  &
=f\left(  x_{k+1}\right)  ^{\intercal}\left(  I-AP_{\bar{d}_{k}}\right)
^{\intercal}A^{-1}\left(  I-AP_{\bar{d}_{k}}\right)  f\left(  x_{k+1}\right)
\\
\left\Vert f\left(  \hat{x}_{k+1}\right)  \right\Vert _{A^{-1}}^{2}  &
=f\left(  x_{k+1}\right)  ^{\intercal}\left(  A^{-0.5}-A^{0.5}P_{\bar{d}_{k}%
}\right)  ^{\intercal}\left(  A^{-0.5}-A^{0.5}P_{\bar{d}_{k}}\right)  f\left(
x_{k+1}\right) \\
\left\Vert f\left(  \hat{x}_{k+1}\right)  \right\Vert _{A^{-1}}^{2}  &
=\left(  A^{-0.5}f\left(  x_{k+1}\right)  \right)  ^{\intercal}\left(
I-A^{0.5}P_{\bar{d}_{k}}A^{0.5}\right)  ^{2}A^{-0.5}f\left(  x_{k+1}\right)  .
\end{align*}
By the properties of projection operators and matrix multiplications, the
minimum eigenvalue of $P$ is zero. Label $\lambda_{\min}$ and $\lambda_{\max}$
are the smallest and largest eigenvalue of $A$, respectively. Since $A$ has no
negative eigenvalues, the minimum eigenvalue of $A^{0.5}P_{\bar{d}_{k}}%
A^{0.5}$ for any $\bar{d}_{k}$ is $\sqrt{\lambda_{\min}}\ast0\ast\sqrt
{\lambda_{\min}}=0$ and its maximum eigenvalue is $\sqrt{\lambda_{\max}}%
\ast1\ast\sqrt{\lambda_{\max}}=\lambda_{\max}<2$. Therefore, the maximum
eigenvalue of $\left(  I-A^{0.5}P_{\bar{d}_{k}}A^{0.5}\right)  ^{2}$ must be
$\max\{\left(  1-0\right)  ^{2},\left(  1-\lambda_{\max}\right)  ^{2}\}=1$.

The ratio
\[
\frac{\left\Vert f\left(  \hat{x}_{k+1}\right)  \right\Vert _{A^{-1}}^{2}%
}{\left\Vert f\left(  x_{k+1}\right)  \right\Vert _{A^{-1}}^{2}}=\frac{\left(
A^{-0.5}f\left(  x_{k+1}\right)  \right)  ^{\intercal}\left(  I-A^{0.5}%
P_{\bar{d}_{k}}A^{0.5}\right)  ^{2}A^{-0.5}f\left(  x_{k+1}\right)  }{\left(
A^{-0.5}f\left(  x_{k+1}\right)  \right)  ^{\intercal}\left(  A^{-0.5}f\left(
x_{k+1}\right)  \right)  }%
\]
is a Rayleigh quotient. By the property of Rayleigh quotients, its maximum
value is the maximum eigenvalue of $\left(  I-A^{0.5}P_{\bar{d}_{k}}%
A^{0.5}\right)  ^{2}$. Therefore $\frac{\left\Vert f\left(  \hat{x}%
_{k+1}\right)  \right\Vert _{A^{-1}}}{\left\Vert f\left(  x_{k+1}\right)
\right\Vert _{A^{-1}}}\leq1$, which completes the proof.
\end{proof}

As mentioned in introduction, with $g\left(  x_{k+1}\right)  $ already
computed, using $\hat{\beta}_{k}$ at step $k$ instead of $\beta_{k}$ involves
computing $g\left(  \hat{x}_{k+1}\right)  $, a second map in the same
iteration $k$. However, if the optimal relaxation parameters tend to be
correlated, using $\hat{\beta}_{k}$ in step $k+1$ only adds two inner
products, a negligible computation cost within the AA algorithm.

\subsection{A linear example with AAmd\label{sec:linear_example_AAnm}}

Consider the same example as in Section \ref{sec:linear_example_AAopt}. Four
different AA implementations will be compared to study the impact of the
initial choice of $\beta_{k}$ on $\hat{\beta}_{k}$ and the impact of using
$\hat{\beta}_{k-1}$ as approximation for $\hat{\beta}_{k}$. The first
implementation is a stationary relaxation parameter $\beta=1$ for reference.
The second is AA where, at each iteration, a default relaxation parameter
$\beta=1$ is used to compute $x_{k+1}$ and $g(x_{k+1})$ and $x_{k+1}$ and
$g(x_{k+1})$ are used to compute $\hat{\beta}_{k}$ to obtain $\hat{x}_{k+1}$.
This second specification will be labeled \textquotedblleft AAmd, $\hat{\beta
}_{k}$ (from 1)\textquotedblright. The third implementation is AA where, at
each iteration, the previous relaxation parameter $\hat{\beta}_{k-1}$ is used
to compute $x_{k+1}$ and $g(x_{k+1})$, which are used to recompute $\hat
{\beta}_{k}$ to obtain $\hat{x}_{k+1}$. This third specification will be
labeled \textquotedblleft AAmd, $\hat{\beta}_{k}$ (from $\hat{\beta}_{k-1}%
$)\textquotedblright. The fourth specification is AA where, at each iteration,
$\hat{x}_{k}$ and $g(\hat{x}_{k})$ are used to compute $\hat{\beta}_{k-1}$,
and $\hat{\beta}_{k-1}$ is used directly at iteration $k$ to compute $\hat
{x}_{k+1}$. It will be labeled \textquotedblleft AAmd, $\hat{\beta}_{k-1}%
$\textquotedblright. Again, $m=8$ is used in all algorithms.%

\begin{figure}[tbp]%
\begin{minipage}{0.48\textwidth}%
%

\raisebox{-0cm}{\includegraphics[
height=7.0973cm,
width=7.0973cm
]%
{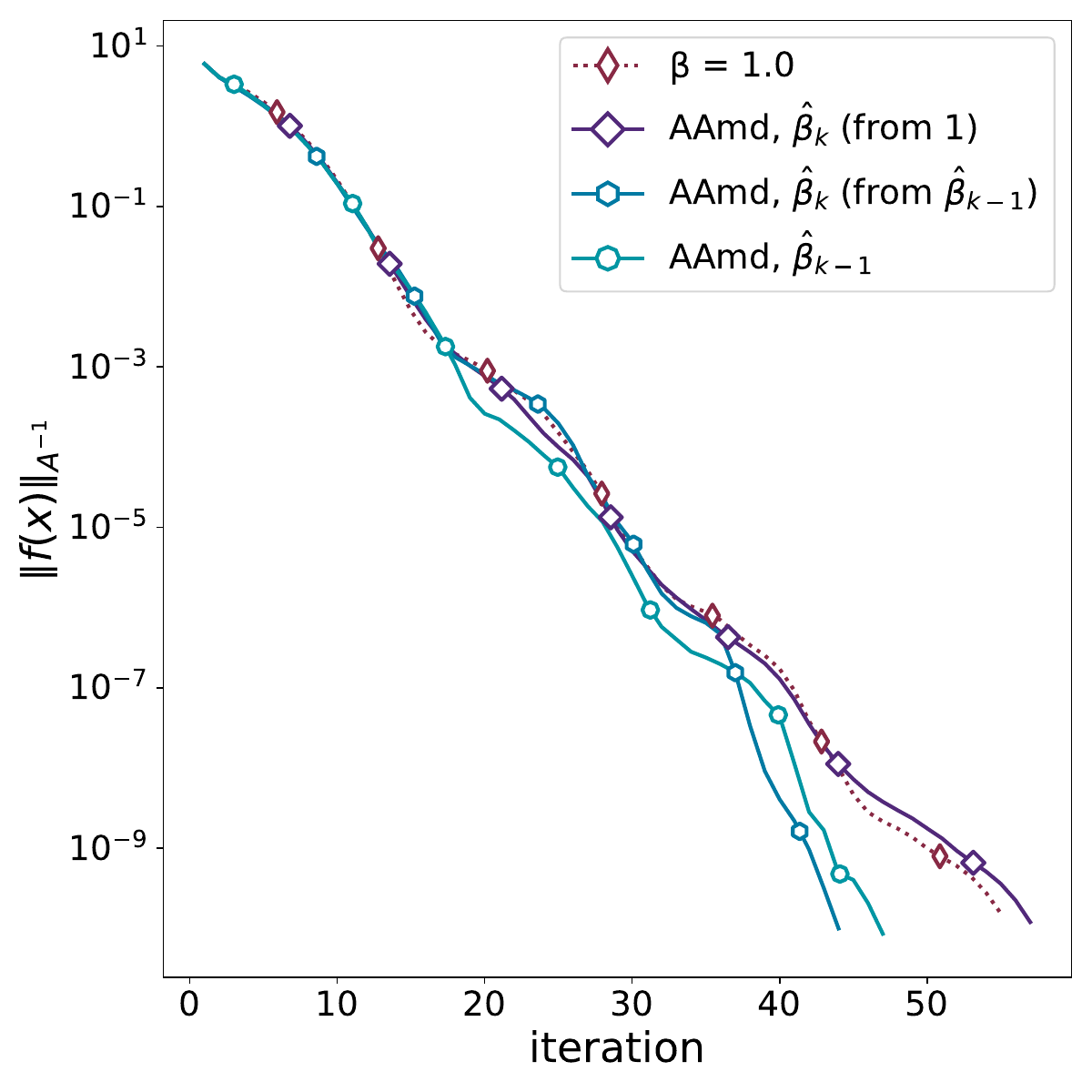}%
}
%

\caption{Residuals (with norm $A^{-1}$), Linear system of equations}%
\label{fig:linear1_AAnm}%
\end{minipage}%
\hfill%
\begin{minipage}{0.48\textwidth}%
\raisebox{-0cm}{\includegraphics[
height=7.0973cm,
width=7.0973cm
]%
{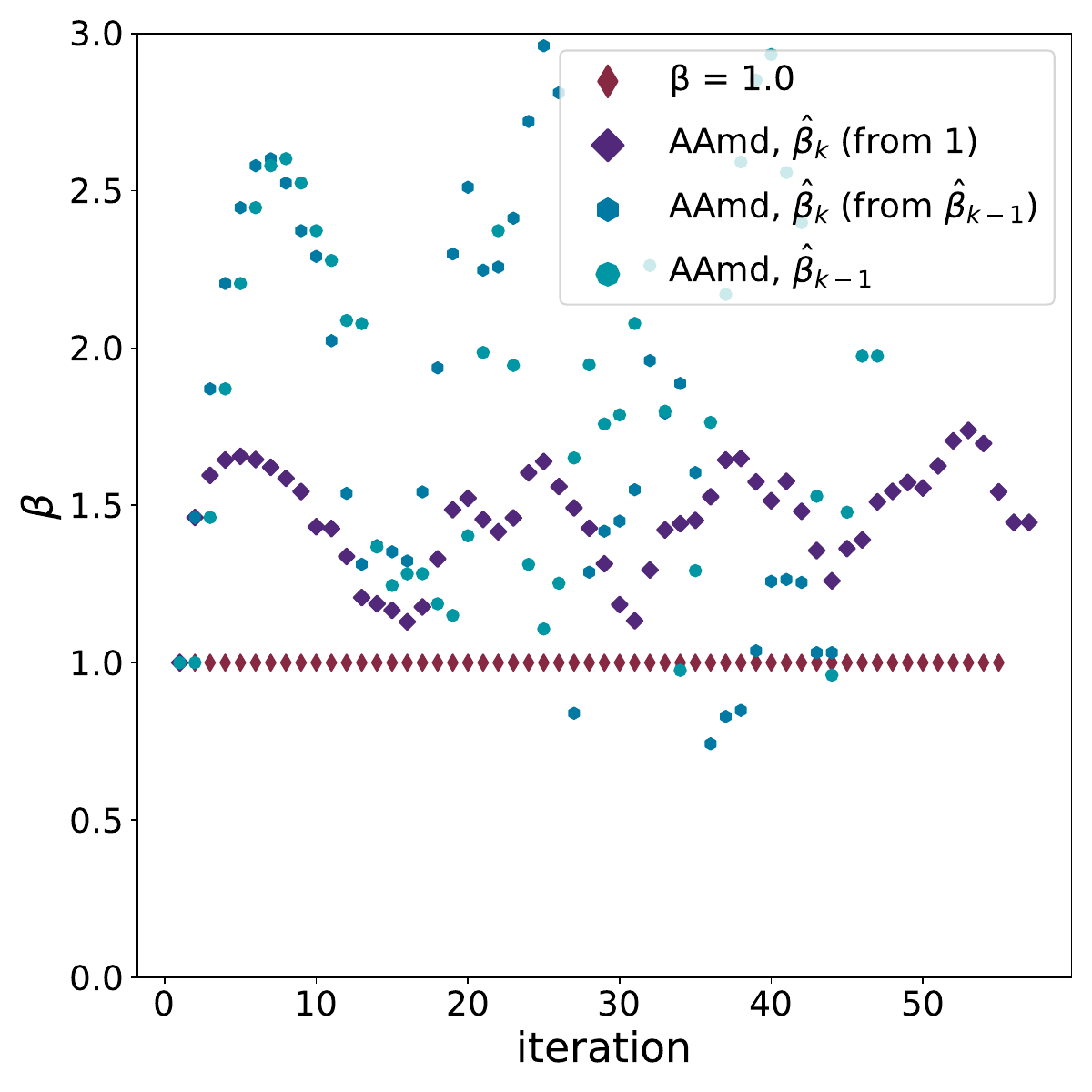}%
}
\caption{Relaxation parameter, linear system of equations}%
\label{fig:linear2_AAnm}%
\end{minipage}%
\end{figure}%

Figure \ref{fig:linear1_AAnm} shows the $A^{-1}$-norm of the residual for each
algorithm. AAmd, $\hat{\beta}_{k}$ (from $\hat{\beta}_{k-1}$) and AAmd,
$\hat{\beta}_{k-1}$ show very modest convergence improvements. Note that with
$m\geq19$, all algorithms would essentially converge in the same number of iterations.

More interestingly, Figure \ref{fig:linear2_AAnm} shows the relaxation
parameter for each AA implementations. Interestingly, most $\hat{\beta}_{k}$
are above 1, and very often above 2. As argued before, they are also obviously
correlated from one iteration to the next. Hence, it makes sense to avoid
computing an extra map by using past information to compute $\hat{\beta}$.
Finally, the default relaxation parameter used to calculate $x_{k+1}$ matters,
as attested by the relatively poor performance of AAmd, $\hat{\beta}_{k}$
(from $1$) compared to AAmd, $\hat{\beta}_{k}$ (from $\hat{\beta}_{k-1}$) and
AAmd, $\hat{\beta}_{k-1}$.

\subsection{Extra regularization for AAmd}

From now on, we only consider AAmd, $\hat{\beta}_{k-1}$ -- where $\hat{\beta
}_{k-1}$ is used as an approximation for $\hat{\beta}_{k}$ to compute
$x_{k+1}$. There is no guarantee that $\hat{\beta}_{k}\approx\hat{\beta}%
_{k-1}$, but a heuristic way of detecting whether it may be the case is by
verifying how much $\hat{\beta}_{k}$ varies between iterations. If it varies
too much, a safer strategy is to fall back on the default relaxation parameter.

Another concern is the fact that a high $\hat{\beta}_{k}$ can lead to an even
higher $\hat{\beta}_{k+1}$, $\hat{\beta}_{k+2}$, etc. In many numerical
experiments shown in Section \ref{sec:Examples}, $\hat{\beta}_{1}$ is set to a
default relaxation parameter (always 1), but $\hat{\beta}_{2},\hat{\beta}%
_{3},...$ sometimes diverge far away from the unit interval, causing worse
convergence than with the default $\beta=1$. A very effective solution to this
problem is to reset $\beta_{k}$ to $1$ (or some other default value between 0
and 1) for one iteration if $\hat{\beta}$ has been above 1 for too many
consecutive iterations.

\subsection{The full AAmd algorithm}

The entire AAmd algorithm is as follows. It takes as input two additional
parameters: $\delta$, the discrepancy allowed between $\hat{\beta}_{k}$ and
$\hat{\beta}_{k-1}$ and $P$, the maximum number of consecutive iterations
where $\hat{\beta}_{k}$ is allowed to be greater than 1 before $\hat{\beta
}_{k}$ is reset to $\beta_{\text{default}}$.

\newpage

\begin{algorithm}%
\label{algo:AAnm}%
Input: a mapping $g:%
\mathbb{R}
^{n}\rightarrow%
\mathbb{R}
^{n}$, a starting point $x_{0}\in%
\mathbb{R}
^{n}$, $1\leq m\leq n$, $P\geq0$,$\delta>0,\beta_{\max}>0\ \ $and
$0<\beta_{\text{default}}\leq\beta_{\max}$.
\end{algorithm}

%

\noindent\makebox[\linewidth]{\rule{\textwidth}{0.4pt}}%
%

\renewcommand{\arraystretch}{1.2}%
%

\noindent
\begin{tabular}
[c]{p{0.6cm}l}%
1 & Set $x_{1}=g(x_{0})$%
\end{tabular}
%

\noindent
\begin{tabular}
[c]{p{0.6cm}l}%
2 & Set $n_{>1}=0$%
\end{tabular}
%

\noindent
\begin{tabular}
[c]{p{0.6cm}l}%
3 & for $k=1,2,...$ until convergence
\end{tabular}
%

\noindent
\begin{tabular}
[c]{p{0.6cm}l}%
4 & \qquad Compute $g\left(  x_{k}\right)  $%
\end{tabular}
%

\noindent
\begin{tabular}
[c]{p{0.6cm}l}%
5 & \qquad Compute $(\alpha_{1}^{(k)},...,\alpha_{m_{k}}^{(k)})$ that solve
\end{tabular}
%

\noindent
\begin{tabular}
[c]{p{0.6cm}l}
& \qquad\quad$\min_{\alpha_{1}^{(k)},...,\alpha_{m_{k}}^{(k)}}\left\Vert
\sum_{i=1}^{m_{k}}\alpha_{i}^{(k)}f\left(  x_{k-m_{k}+i}\right)  \right\Vert $
s.t. $\sum_{i=1}^{m_{k}}\alpha_{i}^{(k)}=1$%
\end{tabular}
%

\noindent
\begin{tabular}
[c]{p{0.6cm}l}%
6 & \qquad Compute $\overline{x}_{k}=\sum_{i=1}^{m_{k}}\alpha_{i}%
^{(k)}x_{k-m_{k}+i}\text{ and }\overline{y}_{k}=\sum_{i=1}^{m_{k}}\alpha
_{i}^{(k)}g\left(  x_{k-m_{k}+i}\right)  $%
\end{tabular}
%

\noindent
\begin{tabular}
[c]{p{0.6cm}l}%
7 & \qquad If $k\geq2$: Compute $\hat{\beta}_{k-1}=\frac{\left\langle
\overline{y}_{k-1}-\overline{x}_{k-1},g\left(  x_{k}\right)  -\overline
{x}_{k-1}\right\rangle }{\left\Vert \overline{y}_{k-1}-\overline{x}%
_{k-1}\right\Vert ^{2}}$%
\end{tabular}
%

\noindent
\begin{tabular}
[c]{p{0.6cm}l}%
8 & \qquad If $k\geq3$ and $|\hat{\beta}_{k-1}-\hat{\beta}_{k-2}|<\delta\ $and
$n_{>1}\leq P$ and $\hat{\beta}_{k-1}>0$%
\end{tabular}
%

\noindent
\begin{tabular}
[c]{p{0.6cm}l}%
9 & \qquad\qquad Set $\beta_{k}=\min(\hat{\beta}_{k-1},\beta_{\max})$%
\end{tabular}
%

\noindent
\begin{tabular}
[c]{p{0.6cm}l}%
10 & \qquad else
\end{tabular}
%

\noindent
\begin{tabular}
[c]{p{0.6cm}l}%
11 & \qquad\qquad Set $\beta_{k}=\beta_{\text{default}}$%
\end{tabular}
%

\noindent
\begin{tabular}
[c]{p{0.6cm}l}%
12 & \qquad end if
\end{tabular}
%

\noindent
\begin{tabular}
[c]{p{0.6cm}l}%
13 & \qquad If $\beta_{k}>1$%
\end{tabular}
%

\noindent
\begin{tabular}
[c]{p{0.6cm}l}%
14 & \qquad\qquad Set $n_{>1}=n_{>1}+1$\\
\multicolumn{1}{l}{15} & \qquad else\\
\multicolumn{1}{l}{16} & \qquad\qquad Set $n_{>1}=0$\\
\multicolumn{1}{l}{17} & \qquad end if\\
\multicolumn{1}{l}{18} & \qquad Set $x_{k+1}=\overline{x}_{k}+\beta
_{k}(\overline{y}_{k}-\overline{x}_{k})$\\
\multicolumn{1}{l}{19} & end for
\end{tabular}
%

\renewcommand{\arraystretch}{1}%
%

\noindent\makebox[\linewidth]{\rule{\textwidth}{0.4pt}}%

The parameters chosen for the experiments below are $\beta_{\text{default}}%
=1$, $\beta_{\max}=3$, $\delta=2,P=10$.

\section{Implementation details for Anderson acceleration}

\subsection{Solving the linear system}

In addition to computing $g$, AA can spend a substantial amount of time
solving the linear system. The problem is customarily formulated as an
unconstrained optimization%
\[
\min_{\gamma_{k}}\left\Vert f\left(  x_{k}\right)  -\mathcal{F}_{k}\gamma
_{k}\right\Vert ,
\]
where $\mathcal{F}_{k}\in%
\mathbb{R}
^{n\times m_{k}}=\left[  f\left(  x_{k}\right)  -f\left(  x_{k-1}\right)
,\cdots,f\left(  x_{k-m_{k}}\right)  -f\left(  x_{k-m_{k}-1}\right)  \right]
$, and $\gamma_{k}\in%
\mathbb{R}
^{m_{k}}$. As suggested in \cite{Walker2011}, it can be solved more quickly by
QR decomposition. New columns can be added from the right of the Q and R
matrices at each iteration and efficiently dropped from the left using Givens
rotation. In the AA implementation used in the numerical section, the QR
decomposition is recomputed anew after 10 rotations to limit the accumulation
of numerical inaccuracies.

A central concern of all Anderson-type acceleration methods is the
conditioning of the linear system. As the algorithm converges, new columns can
be orders of magnitude smaller than old ones and can sometimes be close to
linearly dependent. As in \cite{Walker2011}, this will be addressed by
dropping left-most columns of $\mathcal{F}_{k}$ until the R matrix has a
reasonable condition number. In the PDE estimation, the upper limit for the
condition number was set to $10^{12}$ while for the EM algorithm application,
it was $10^{5}$.

Other methods for addressing ill-conditioning and adjusting $m_{k}$ have been
suggested in \cite{Fang2009}, \cite{Pollock2023} and \cite{Chen2024}.

\subsection{Restarts and composite AA\label{sec:Restarts_and_composite_AA}}

To limit ill-conditioning and the size of the linear system to be solved, Fang
and Saad \cite{Fang2009} suggested restarting the algorithm from the last
iterate and ignoring past directions. Pratapa and Suryanarayana
\cite{Pratapa2015} and Henderson and Varadhan \cite{Henderson2019} made
similar points in the context of Pulay mixing and AA.

A cousin of this idea is the composite AA, explored by Chen and Vuik in 2022
\cite{Chen2022}. After one AA iteration, instead of using the iterate $x_{k}$
directly to compute the map $g(x_{k})$, they propose using $x_{k}$ as the
starting point for a second AA (lasting only one or two iterations), and
feeding the result of this second (inner-loop) AA back in the original
(outer-loop) AA.

Both ideas were tried on all problems. Whilst periodic restarts did not
reliably improve performances, composite AAmd with a one-iteration AA for the
inner loop showed very good results for the PDE problem. Hence, composite AA
with a one-iteration AA with $\beta=1$ in the inner loop will be included in
the set of specifications to test.

\section{Applications\label{sec:applications}}

A Poisson PDE applications and two EM algorithm applications were used as
benchmarks. They are sufficiently challenging to estimate and require enough
iterations to create visible differences in performance between different AA
implementations. They also offer a good variety of number of parameters and
mapping computation time.

The EM algorithm is commonly used to fit statistical models with missing data
or latent variables to estimate the parameters of underlying unobserved
distributions. It consists of two steps. An expectation step takes the model
parameters as given and updates the parameters of the unobserved data via
Bayes' rule. Then, the likelihood of the observed data is maximized, taking
the unobserved distributions as given. The EM algorithm is usually very stable
and always converges, although it can sometimes be to a saddle point instead
of a maximum. It is also notoriously slow, making it a prime candidate for
acceleration. Both EM application were adapted from the R code used in
\cite{Henderson2019}.

\subsection{The Bratu problem\label{sec:Bratu}}

The standard Liouville-Bratu-Gelfand equation is a nonlinear version of the
Poisson equation, described as%

\[
\Delta u+\lambda e^{u}=0,
\]
where $u$ is a function $\left(  x,y\right)  \in\mathcal{D}=[0,1]^{2}$ and
$\lambda$ is a constant physical parameter. It is a popular application for
benchmarking new fixed-point acceleration methods (see \cite{Walker2011} and
\cite{Fang2009}, for example). Dirichlet boundary conditions are applied such
that $u(x,y)=0$ on the boundary of $\mathcal{D}$. It is solved using the
inverse of the discrete Laplace operator as preconditioner as in
\cite{Chen2024}. The mapping is%

\[
x_{i}^{(k+1)}=x_{i}^{(k)}+(b_{i}-A_{i}x+\lambda e^{x_{i}^{(k)}})/A_{i,i}\quad
i=1,...,50^{2},
\]

where $A$ is the Laplace operator, $A_{i,i}$ is its row $i$ and column $i$'s
entry, and $A_{i}$ is the entire row $i$. In the experiments, $\lambda$ was
set to $6$ and a centered-difference discretization on a $50\times50$ grid was used.

\subsection{The EM algorithm for a proportional hazard model with interval
censoring}

Proportional hazard models are commonly used in medical and social studies,
and censored data is a frequent occurrence which complicates their estimation.
Wang et al. \cite{Wang2016} proposed using the EM algorithm to estimate a
semiparametric proportional hazard model with interval censoring. Their
estimation is a two-stage data augmentation with latent Poisson random
variables and a monotone spline to represent the baseline hazard function. The
algorithm is light and simple to implement (see \cite{Wang2016} for details),
yet may benefit from acceleration.

The likelihood of an individual observation is
\[
L(\delta_{1},\delta_{2},\delta_{3},\mathbf{x})=F(R|\mathbf{x})^{\delta_{1}%
}\{F(R|\mathbf{x})-F(L|\mathbf{x})\}^{\delta_{2}}\{1-F(L|\mathbf{x}%
)\}^{\delta_{3}},
\]
where $\delta_{1}$, $\delta_{2}$, and $\delta_{3}$ represent right-,
interval-, or left-censoring, respectively.

Synthetic test data is produced as follows. The baseline hazard function is
modeled as a six-parameter I-spline and generated as $\Lambda_{0}%
(t)=\log(1+t)+t^{1/2}$. Failure times $T$ are generated from a distribution
$F(t,\mathbf{x})=1-\exp\{-\Lambda_{0}(t)\exp(\mathbf{x}^{\intercal}\beta)\}$.
The covariates $\mathbf{x}$ are $x_{1},x_{2}\sim N(0,0.5^{2})$ and
$x_{3},x_{4}\sim$Bernoulli$(0.5)$, for a total of 10 parameters to estimate.
For each subject, censoring is simulated by generating a $Y\sim$%
Exponential$(1)$ distribution and setting $(L,R)=(Y,\infty)$ if $Y\leq T$ or
$(L,R)=(0,Y)$ if $Y>T$. The sample size was 2000 individuals.

During the estimations, monitoring the value of the likelihood was not
necessary for convergence.

\subsection{EM algorithm for admixed populations}

When associating health outcomes with specific genes, a recurring confounding
factor is population stratification, the clustering of genes within population
subgroups. In \cite{Alexander2009}, Alexander, Novembre and Lange put forward
a new algorithm called ADMIXTURE to identify latent subpopulations from
genomics data using the EM algorithm.

To test the algorithm, datasets are simulated as follows. Individual
$i\in1,...,n$ are assumed to be part of $K$ distinct ancestral groups in
various proportions. Each has a pair of alleles ($a_{i,j}^{1},a_{i,j}^{2}$) at
the marker $j\in1,...,J$ with major or minor frequencies recorded in a
variable $X$. For individual $i$ and marker $j$, $X_{i,j}=0$ if both minor
alleles are minor, $X_{i,j}=1$ if one is minor and one is major, and
$X_{i,j}=2$ if both are major. The probability of observing $X_{i,j}$ is
determined by ancestry-specific parameters $f_{k,j}$, the frequency of minor
alleles at the marker $j$ in the ancestral population $k\in1...K$, and the
parameter $q_{i,k}$ representing the unobserved proportion of ancestry of
individual $i$ from group $k$. The log-likelihood function is%
\begin{equation}
L\left(  \mathbf{F},\mathbf{Q}\right)  =\sum_{i=1}^{n}\sum_{j=1}^{J}\left\{
X_{i,j}\log\left(  \sum_{k=1}^{K}q_{i,k}f_{k,j}\right)  +\left(
2-X_{i,j}\right)  \log\left(  \sum_{k=1}^{K}q_{i,k}\left(  1-f_{k,j}\right)
\right)  \right\}  , \label{ancestry_likelihood}%
\end{equation}

where $\mathbf{F}^{K\times J}$ and $\mathbf{Q}^{n\times K}$ are matrices with
entries $f_{k,j}$ and $q_{i,k}$,respectively. To restrict probabilities to the
unit interval during the estimation, they are modeled as transformations from
unbounded parameters $u_{k,j},v_{i,k}$: $f_{k,j}=1/(1+e^{-u_{k,j}})$ and
$q_{i,k}=e^{v_{i,k}}/\sum_{k}e^{v_{i,k}}$. New random values for $X$ and new
starting values are generated for each draw with parameters $K=3$, $J=100$,
and $n=150$, for a total of $3\times(100+150)=750$ parameters to estimate.

For AA to converge, it was necessary to monitor the likelihood value
(\ref{ancestry_likelihood}) and fall back to the last EM iteration in case an
AA iteration lead to a worse likelihood value. Also, since convergence was
slow, the stopping criterion was set to $\left\Vert f(x)\right\Vert
\leq10^{-4}$.

\subsection{Example Bratu problem\label{sec:Examples}}

\subsubsection{Non-composite AA\label{sec:Bratu_non_composite}}

This section compares AAopt1 and AAmd to AAopt0 and AA with stationary
relaxation parameters of $\beta=1$ and $\beta=0.5$ for the Bratu problem with
a $x_{0}=\mathbf{0}$ starting point. Hereafter, AAmd refers to Algorithm
\ref{algo:AAnm} with $\hat{\beta}$ capped at $\beta_{\max}$ and resets to
$\hat{\beta}=1$. To show the impact of these regularizations, the examples
also show the performance of AAmd without bounds on $\hat{\beta}$ or resets,
labeled \textquotedblleft AAmd (no reg.)\textquotedblright. All algorithms
used a maximum of $m=16$ lags. AAopt0 was implemented as in \cite{Chen2024},
with $\beta_{k}$ set to $0.5$ whenever the optimal relaxation parameter
$\beta_{k}^{\ast}$ was zero or fell outside the unit interval.%

\begin{figure}[tbp]%
\begin{minipage}{0.48\textwidth}%
%

\raisebox{-0cm}{\includegraphics[
height=7.0973cm,
width=7.0973cm
]%
{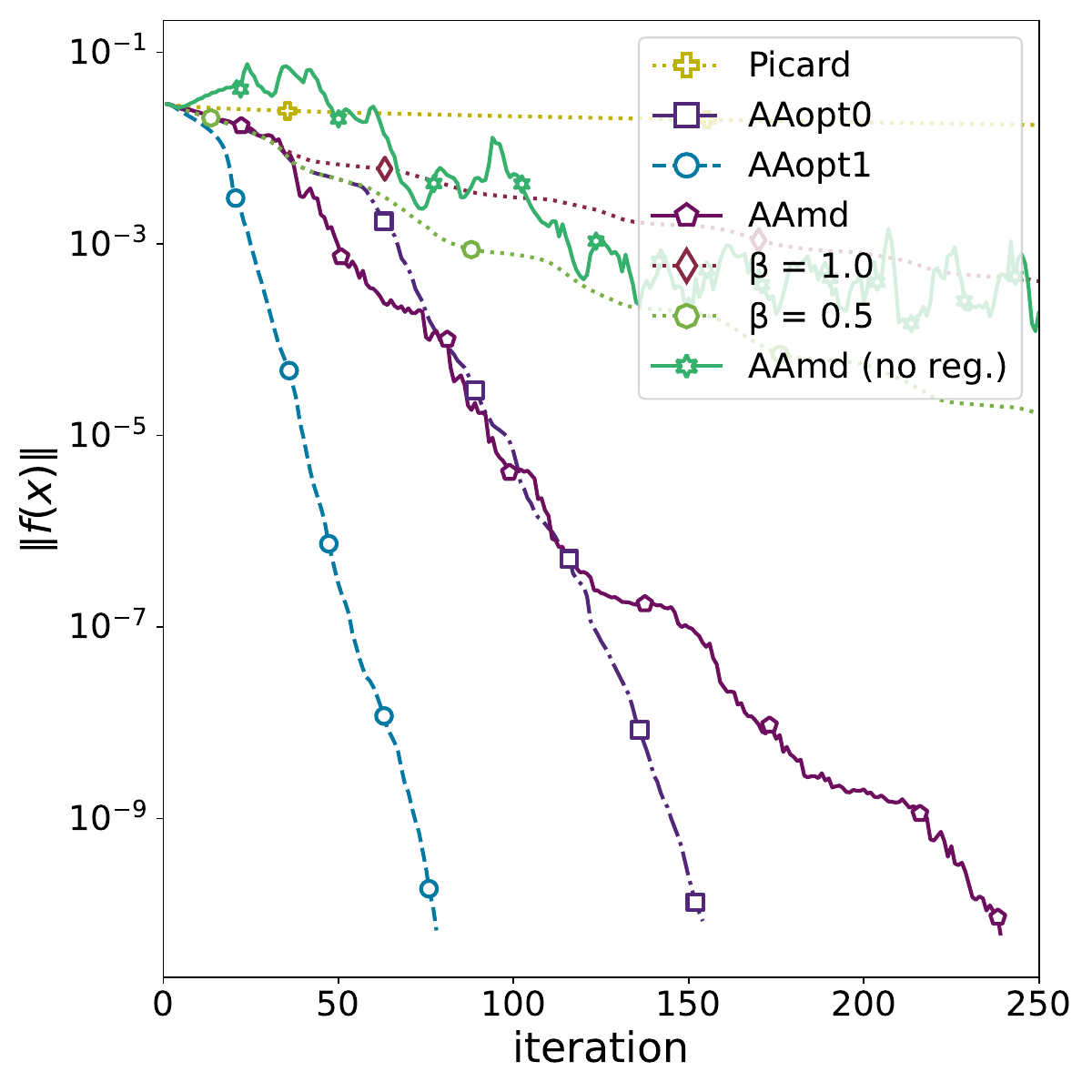}%
}
\caption{Residual norm, Bratu problem, m=16}%
\label{fig:bratu1_resid}%
\end{minipage}%
\hfill%
\begin{minipage}{0.48\textwidth}%
\raisebox{-0cm}{\includegraphics[
height=7.0973cm,
width=7.0973cm
]%
{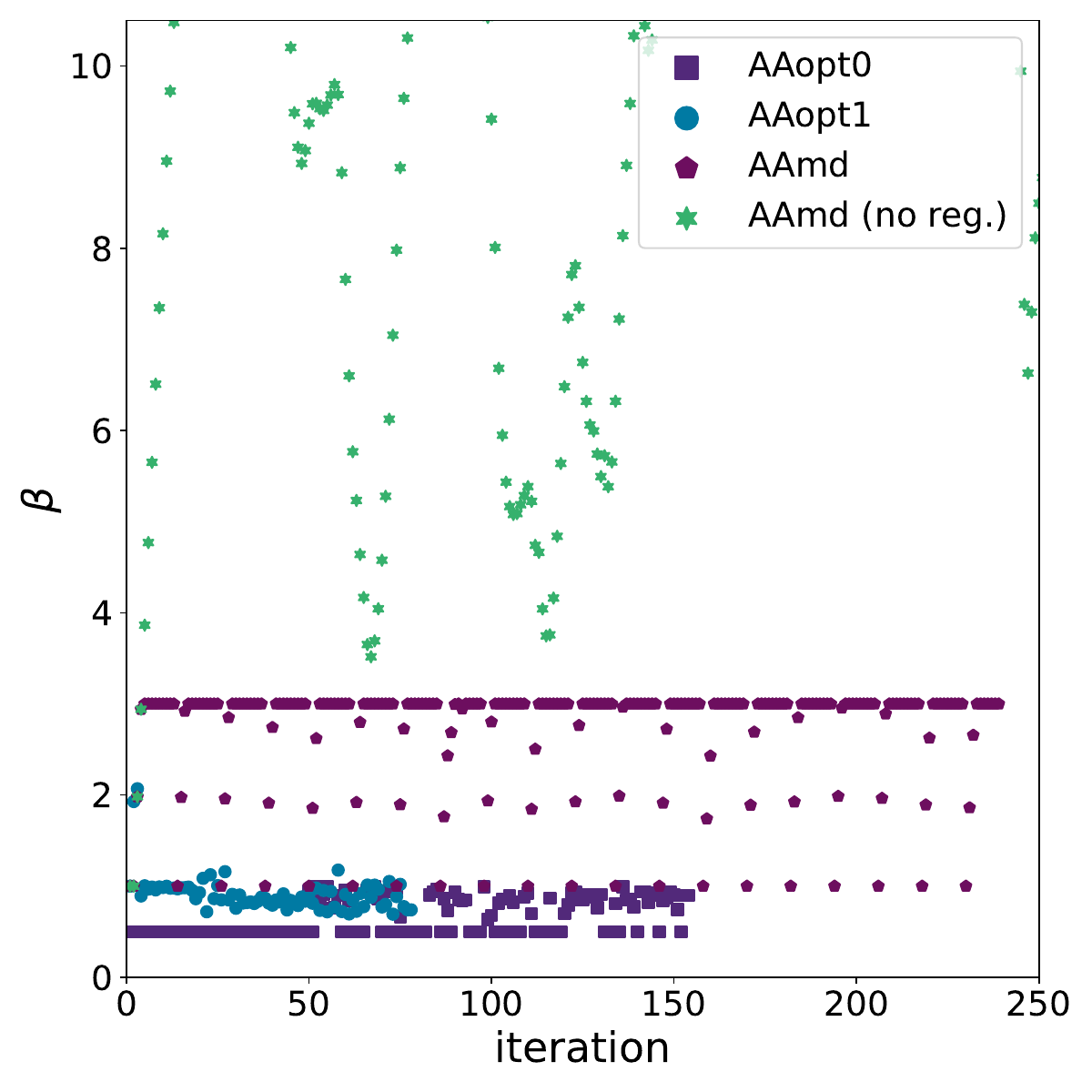}%
}
\caption{Relaxation parameters, Bratu problem, m=16}%
\label{fig:bratu1_relax}%
\end{minipage}%
\end{figure}%

Figure \ref{fig:bratu1_resid} shows the residuals for the Bratu problem for
non-composite versions of AAopt0, AAopt1, AAmd, and AAmd (no reg.). Like in
the linear model of Section \ref{sec:linear_example_AAopt}, AAopt1 needed the
fewest iterations to converge. However, AAmd's results are promising knowing
that requires a single mapping per iteration.

Figure \ref{fig:bratu1_relax} shows the corresponding relaxation parameters.
Looking at AAmd (no reg.), it is striking how without constraints or restarts,
$\hat{\beta}_{k}$ takes extremely large values and converges as slowly as AA
with stationary relaxation. With regularization, AAmd performs much better
than AA with constant relaxation.

\subsubsection{Composite AA}

The Bratu problem is estimated with the same specifications as in Section
\ref{sec:Bratu_non_composite}, except that AA is now composite with an AA1
inner loop (identified with a c, see Section
\ref{sec:Restarts_and_composite_AA}).%

\begin{figure}[tbp]%
\begin{minipage}{0.48\textwidth}%
\raisebox{-0cm}{\includegraphics[
height=7.0973cm,
width=7.0973cm
]%
{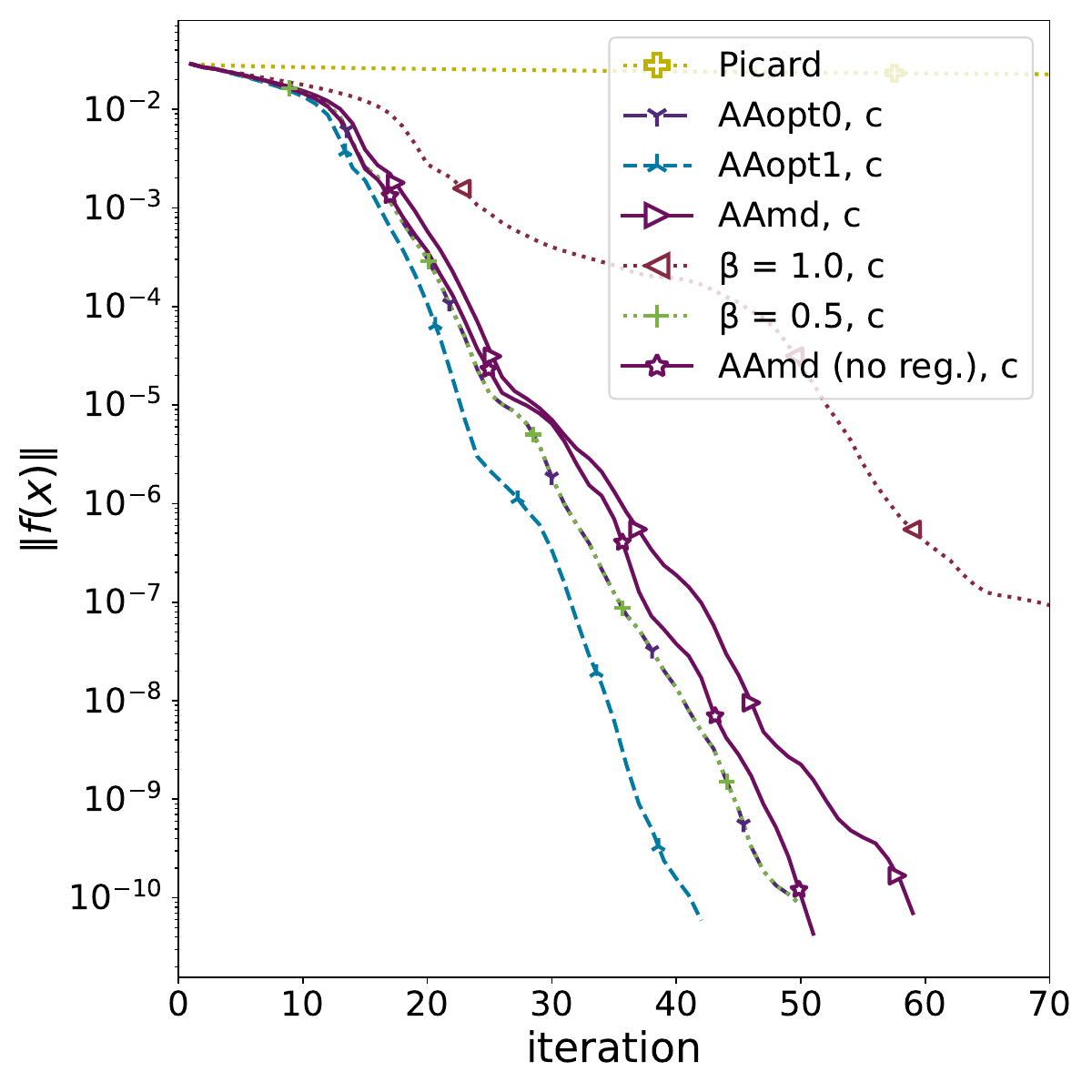}%
}
\caption{Residual norm, Bratu problem, m=16 with, composite AA}%
\label{fig:bratu2_resid}%
%

\end{minipage}%
\hfill%
\begin{minipage}{0.48\textwidth}%
\raisebox{-0cm}{\includegraphics[
height=7.0973cm,
width=7.0973cm
]%
{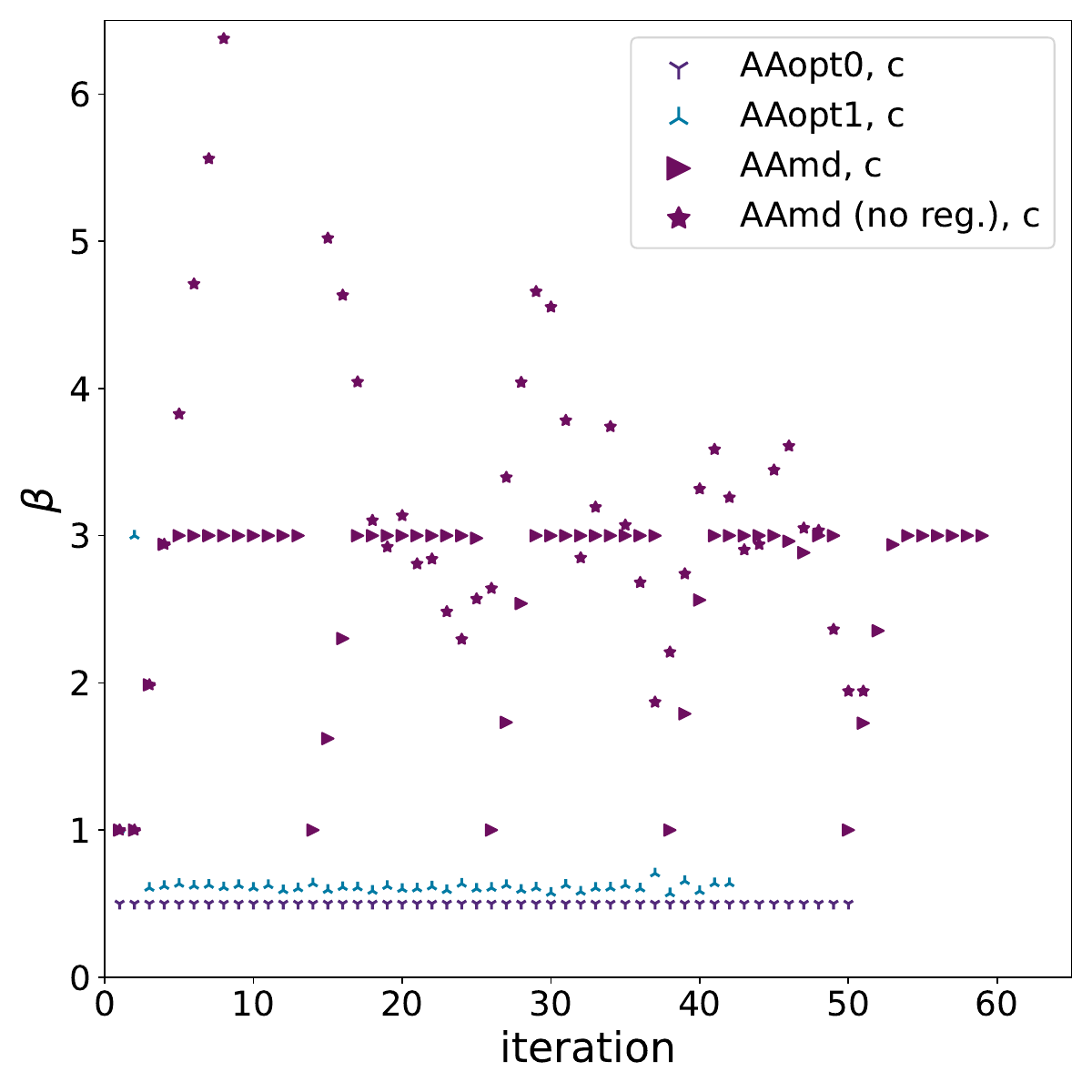}%
}
\caption{Relaxation parameters, Bratu problem, m=16, composite AA
}%
\label{fig:bratu2_relax}%
\end{minipage}%
\end{figure}%

Figures \ref{fig:bratu1_resid} show the residuals for the Bratu problem with
composite versions of AA and Figure \ref{fig:bratu1_relax} shows the
corresponding relaxation parameters. The horizontal axis refers to the number
of outer loop iterations. AAopt1, c converges in the fewest iterations,
although the difference with other AA algorithms is relatively small,
especially considering the fact that it requires 6 maps per iteration. AA
implementations that require only 3 mappings per iteration such as AAmd, c and
AA, c with constant $\beta=0.5$ also converged reasonably quickly.

\subsection{Experiments\label{sec:Experiments}}

Based on the previous example, the most promising AA implementations to
investigate further are AAopt1, AAmd and AAmd, c. Additionally, AAopt1\_4 and
AAopt1\_16, which update $\beta^{\ast}$ every 4 or every 16 iterations, are
also studied. These implementations are again compared with stationary AA with
$\beta=1$ and $\beta=0.5$, as well as their composite versions.

To select the optimal $m$ for each AA algorithm in each experiment, each
algorithm was implemented with $m\in\left\{  2,4,8,16,32,64\right\}  $ for 500
draws. The fastest in terms of computation speed (at the 0.75 quantile to
favor robustness) was selected. For the EM algorithm for the proportional
hazard model with interval censoring, the maximum $m=10$ was clearly the best
choice for all algorithms.

A total of 5000 draws were generated for the Bratu problem and the EM for the
proportional hazard model with interval censoring. To reduce simulation time,
only 1000 draws were calculated for the EM algorithm for admixed populations.
The stopping criterion was $\left\Vert f\left(  x_{k}\right)  \right\Vert
\leq10^{-8}$ for all applications, except the EM algorithm for admixed
populations which used $\left\Vert f\left(  x_{k}\right)  \right\Vert
\leq10^{-4}$. All algorithms that did not converge were stopped after 10 000 mappings.

Computation times are presented using the performance profiles of Dolan and
Mor\'{e} \cite{Dolan2002}. They show at which frequency each algorithm's time
was within a certain factor of the fastest algorithm for each draw. The
99\%\ confidence intervals for the median number of iterations (outer-loop
iterations for composite AA), mappings, and computation times are also
reported in tables, along with convergence rates. All computations were
single-threaded, performed on Julia 1.10.4 \cite{bezanson2017}, with a 13th
Gen Intel(R) Core(TM) i9-13900HX 2.20 GHz CPU running the Windows subsystem
for Linux.%

\begin{figure}[tbp]%
\begin{center}%
\begin{minipage}{0.48\textwidth}%
\raisebox{-0cm}{\includegraphics[
height=7.0973cm,
width=7.0973cm
]%
{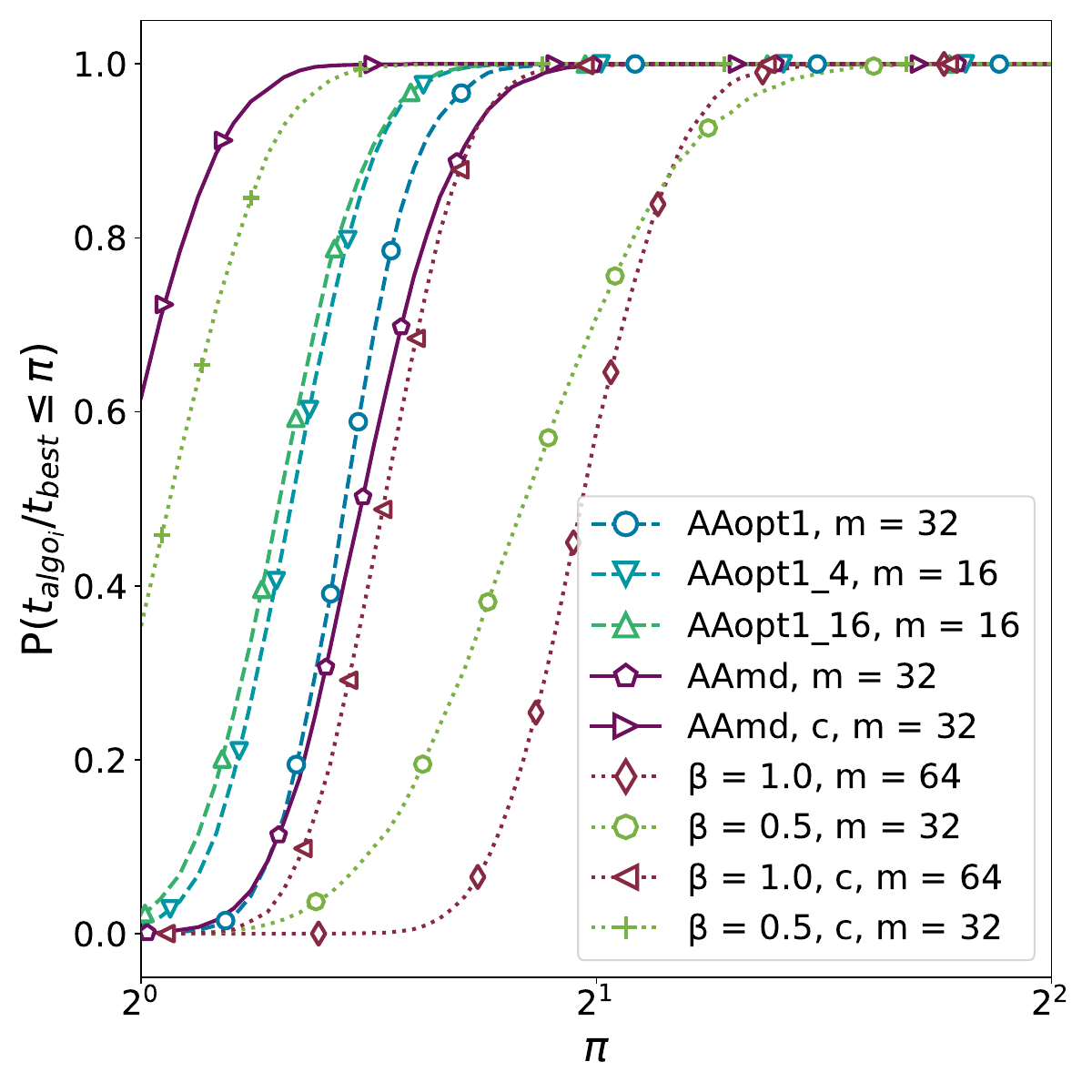}%
}
\caption{Performance profiles for the Bratu problem}%
\label{fig:bratu}%
\end{minipage}%
\end{center}%
\end{figure}%

For the Bratu problem, the starting values $x_{0}^{(i)}$ were drawn from
U(0,1) distributions. The results are reported in Figure \ref{fig:bratu} and
Table \ref{tab:bratu}. The composite AAmd was the fastest more than 60\% of
the time, requiring a median of 67 iterations and 199 maps to converge. The
runner-up was the composite AA with constant $\beta=0.5$. Interestingly,
non-composite AAmd was also faster than both non-composite AA with constant relaxation.%

\begin{table}[!htbp]%
\begin{minipage}{\textwidth}%
\caption{Median performances: Bratu problem}%
\label{tab:bratu}%
\begin{center}%
\begin{tabular}
[c]{lrrrr}\hline
&  &  &  & \vspace{-0.4cm}\\
Algorithm & \multicolumn{1}{c}{Iterations} & \multicolumn{1}{c}{\ \ \ Maps
\ \ \ } & \multicolumn{1}{c}{\ \ Time (ms) \ \ \ } &
\multicolumn{1}{c}{Converged\vspace{0.1cm}}\\\hline
& \multicolumn{1}{c}{} & \multicolumn{1}{c}{} & \multicolumn{1}{c}{} &
\multicolumn{1}{c}{\vspace{-0.4cm}}\\
AAopt1, $m=32$ & \multicolumn{1}{c}{(89, 89)} & \multicolumn{1}{c}{(263, 263)}
& \multicolumn{1}{c}{(148.22, 149.19)} & \multicolumn{1}{c}{1}\\
AAopt1\_4, $m=16$ & \multicolumn{1}{c}{(161, 162)} & \multicolumn{1}{c}{(241,
242)} & \multicolumn{1}{c}{(136.73, 137.67)} & \multicolumn{1}{c}{1}\\
AAopt1\_16, $m=16$ & \multicolumn{1}{c}{(203, 204)} & \multicolumn{1}{c}{(229,
230)} & \multicolumn{1}{c}{(134.28, 135.47)} & \multicolumn{1}{c}{1}\\
AAmd, $m=32$ & \multicolumn{1}{c}{(218, 219)} & \multicolumn{1}{c}{(218, 219)}
& \multicolumn{1}{c}{(152.02, 153.4)} & \multicolumn{1}{c}{1}\\
AAmd, c, $m=32$ & \multicolumn{1}{c}{(67, 67)} & \multicolumn{1}{c}{(199,
199)} & \multicolumn{1}{c}{(112.12, 112.82)} & \multicolumn{1}{c}{1}\\
AA, $\beta=1.0,$ $m=64$ & \multicolumn{1}{c}{(223, 224)} &
\multicolumn{1}{c}{(223, 224)} & \multicolumn{1}{c}{(212.83, 214.67)} &
\multicolumn{1}{c}{1}\\
AA, $\beta=0.5,$ $m=32$ & \multicolumn{1}{c}{(273, 278)} &
\multicolumn{1}{c}{(273, 278)} & \multicolumn{1}{c}{(194.53, 197.93)} &
\multicolumn{1}{c}{1}\\
AA, $\beta=1.0$, c$,m=64$ & \multicolumn{1}{c}{(87, 87)} &
\multicolumn{1}{c}{(259, 259)} & \multicolumn{1}{c}{(157.77, 158.96)} &
\multicolumn{1}{c}{1}\\
AA, $\beta=0.5$, c$,m=32$ & \multicolumn{1}{c}{(70, 71)} &
\multicolumn{1}{c}{(208, 211)} & \multicolumn{1}{c}{(116.79, 117.9)} &
\multicolumn{1}{c}{1\vspace{0.1cm}}\\\hline
&  &  &  & \vspace{-0.4cm}\\
\multicolumn{5}{l}{Note: 99\% conf. interval for the median. 2500
parameters.\vspace{0.1cm}}\\\hline
\end{tabular}%
\end{center}%
\end{minipage}%
\end{table}%
%

\begin{figure}[tbp]%
\begin{minipage}{0.48\textwidth}%
%

\raisebox{-0cm}{\includegraphics[
height=7.0973cm,
width=7.0973cm
]%
{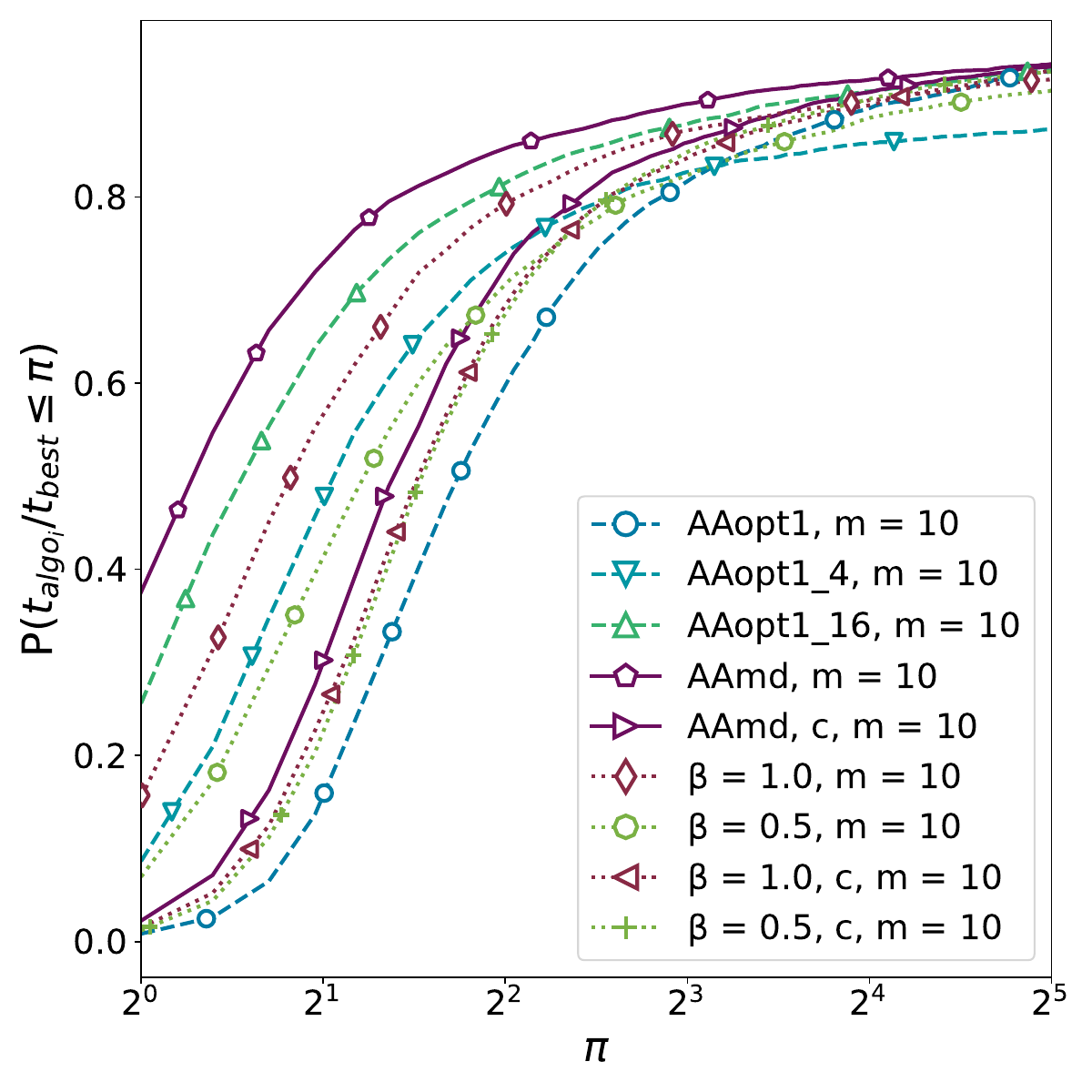}%
}
\caption
{Performance profiles for the EM algorithm for a proportional hazard model with interval censoring}%
\label{fig:censPHEM}%
\end{minipage}%
\hfill%
\begin{minipage}{0.48\textwidth}%
\raisebox{-0cm}{\includegraphics[
height=7.0973cm,
width=7.0973cm
]%
{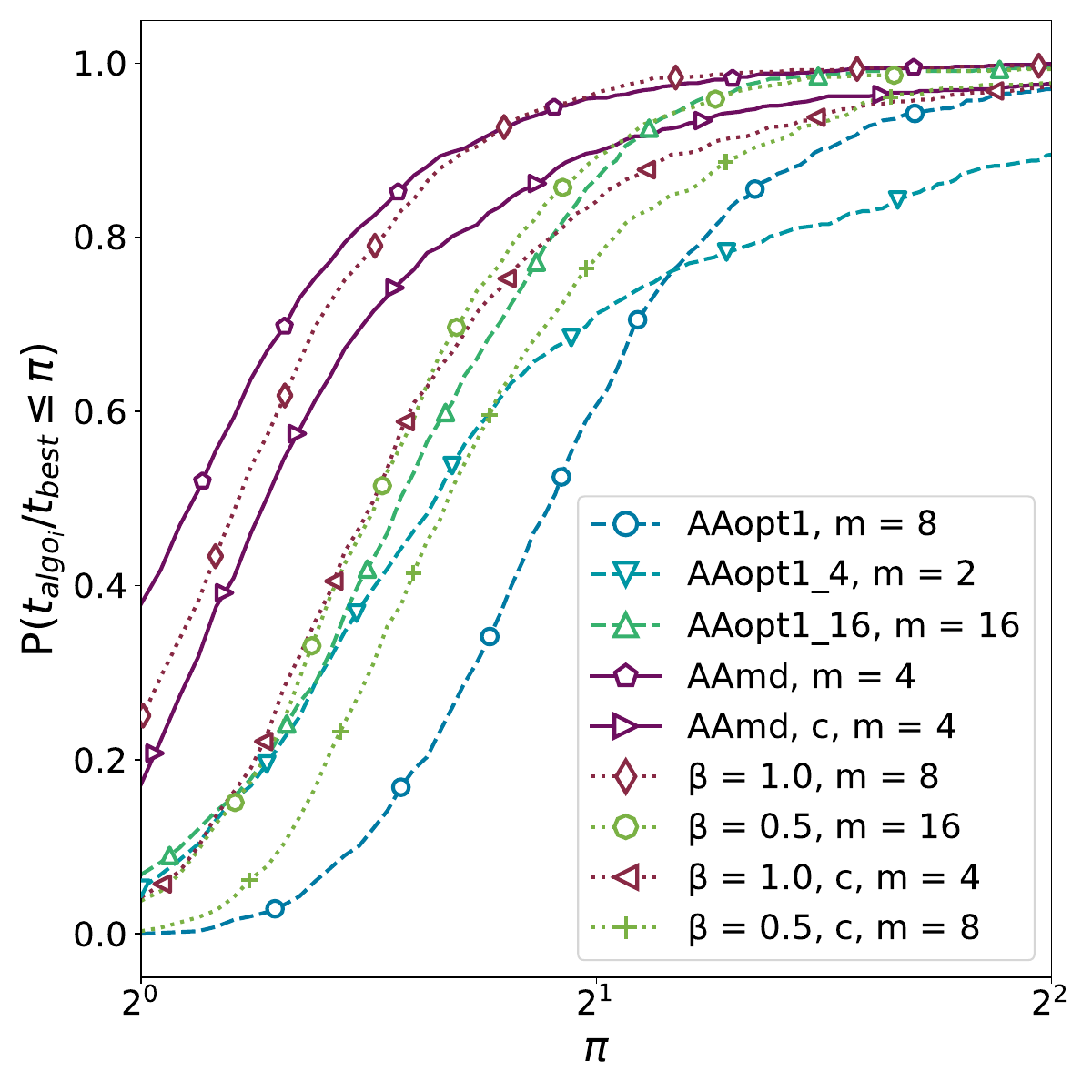}%
}
\caption{Performance profiles for the EM algorithm for admixed populations \newline
}%
\label{fig:ancestry}%
\end{minipage}%
\end{figure}%

The results of the numerical experiments with the EM algorithm for a
proportional hazard model with interval censoring are shown in Figure
\ref{fig:censPHEM} and Table \ref{tab:censPHEM}. AAmd was the fastest to
converge, with a median number of iterations and mapping evaluations of
approximately 99. AAopt1\_16 was a close second. Contrary to the Bratu
problem, composite AA clearly did not benefit from the AA1 inner loop compared
to non-composite AA. Still among composite AA, AAmd did slightly outperformed
those with constant relaxation.%

\begin{table}[!htbp]%
\begin{minipage}{\textwidth}%
\caption
{Median performances: EM algorithm for a proportional hazard model with interval censoring}%
\label{tab:censPHEM}%
\begin{center}%
\begin{tabular}
[c]{lrrrr}\hline
&  &  &  & \vspace{-0.4cm}\\
Algorithm & \multicolumn{1}{c}{Iterations} & \multicolumn{1}{c}{Maps} &
\multicolumn{1}{c}{Time (ms)} & \multicolumn{1}{c}{Converged\vspace{0.1cm}%
}\\\hline
& \multicolumn{1}{c}{} & \multicolumn{1}{c}{} & \multicolumn{1}{c}{} &
\multicolumn{1}{c}{\vspace{-0.4cm}}\\
AAopt1, $m=10$ & \multicolumn{1}{c}{(87, 93)} & \multicolumn{1}{c}{(257, 275)}
& \multicolumn{1}{c}{(71.59, 76.41)} & \multicolumn{1}{c}{0.962}\\
AAopt1\_4, $m=10$ & \multicolumn{1}{c}{(102, 110)} & \multicolumn{1}{c}{(152,
164)} & \multicolumn{1}{c}{(43.71, 47.56)} & \multicolumn{1}{c}{0.894}\\
AAopt1\_16, $m=10$ & \multicolumn{1}{c}{(99, 107)} & \multicolumn{1}{c}{(113,
121)} & \multicolumn{1}{c}{(32.8, 35.14)} & \multicolumn{1}{c}{0.96}\\
AAmd, $m=10$ & \multicolumn{1}{c}{(96, 102)} & \multicolumn{1}{c}{(96, 102)} &
\multicolumn{1}{c}{(28.43, 30.28)} & \multicolumn{1}{c}{0.958}\\
AAmd, c, $m=10$ & \multicolumn{1}{c}{(68, 72)} & \multicolumn{1}{c}{(202,
214)} & \multicolumn{1}{c}{(56.84, 60.03)} & \multicolumn{1}{c}{0.97}\\
AA, $\beta=1.0$, $m=10$ & \multicolumn{1}{c}{(129, 139)} &
\multicolumn{1}{c}{(129, 139)} & \multicolumn{1}{c}{(37.98, 40.41)} &
\multicolumn{1}{c}{0.952}\\
AA, $\beta=0.5$, $m=10$ & \multicolumn{1}{c}{(169, 184)} &
\multicolumn{1}{c}{(169, 184)} & \multicolumn{1}{c}{(48.79, 52.8)} &
\multicolumn{1}{c}{0.937}\\
AA, $\beta=1.0$, c, $m=10$ & \multicolumn{1}{c}{(75, 79)} &
\multicolumn{1}{c}{(223, 235)} & \multicolumn{1}{c}{(61.79, 65.68)} &
\multicolumn{1}{c}{0.962}\\
AA, $\beta=0.5$, c, $m=10$ & \multicolumn{1}{c}{(76, 80)} &
\multicolumn{1}{c}{(226, 239)} & \multicolumn{1}{c}{(63.13, 67.07)} &
\multicolumn{1}{c}{0.964}\\\hline
&  &  &  & \vspace{-0.4cm}\\
\multicolumn{5}{l}{Note: 99\% conf. interval for the median. 10
parameters.\vspace{0.1cm}}\\\hline
\end{tabular}%
\end{center}%
\end{minipage}%
\end{table}%
%

\begin{table}[!htbp]%
\begin{minipage}{\textwidth}%
\caption{Median performances: EM algorithm for admixed populations}%
\label{tab:ancestry}%
\begin{center}%
\begin{tabular}
[c]{lrrrr}\hline
&  &  &  & \vspace{-0.4cm}\\
Algorithm & \multicolumn{1}{c}{Iterations} & \multicolumn{1}{c}{Maps} &
\multicolumn{1}{c}{Time (s)} & \multicolumn{1}{c}{Converged\vspace{0.1cm}%
}\\\hline
& \multicolumn{1}{c}{} & \multicolumn{1}{c}{} & \multicolumn{1}{c}{} &
\multicolumn{1}{c}{\vspace{-0.4cm}}\\
AAopt1, $m=8$ & \multicolumn{1}{c}{(245, 253)} & \multicolumn{1}{c}{(731,
755)} & \multicolumn{1}{c}{(4.14, 4.28)} & \multicolumn{1}{c}{1}\\
AAopt1\_4, $m=2$ & \multicolumn{1}{c}{(355, 374)} & \multicolumn{1}{c}{(533,
560)} & \multicolumn{1}{c}{(3.4, 3.68)} & \multicolumn{1}{c}{1}\\
AAopt1\_16, $m=16$ & \multicolumn{1}{c}{(434, 454)} & \multicolumn{1}{c}{(488,
512)} & \multicolumn{1}{c}{(3.38, 3.55)} & \multicolumn{1}{c}{1}\\
AAmd, $m=4$ & \multicolumn{1}{c}{(363, 379)} & \multicolumn{1}{c}{(363, 379)}
& \multicolumn{1}{c}{(2.55, 2.65)} & \multicolumn{1}{c}{1}\\
AAmd, c, $m=4$ & \multicolumn{1}{c}{(161, 166)} & \multicolumn{1}{c}{(481,
496)} & \multicolumn{1}{c}{(2.76, 2.86)} & \multicolumn{1}{c}{0.999}\\
AA, $\beta=1.0$, $m=8$ & \multicolumn{1}{c}{(377, 389)} &
\multicolumn{1}{c}{(377, 389)} & \multicolumn{1}{c}{(2.67, 2.77)} &
\multicolumn{1}{c}{1}\\
AA, $\beta=0.5$, $m=16$ & \multicolumn{1}{c}{(449, 461)} &
\multicolumn{1}{c}{(449, 461)} & \multicolumn{1}{c}{(3.21, 3.35)} &
\multicolumn{1}{c}{1}\\
AA, $\beta=1.0$, c, $m=4$ & \multicolumn{1}{c}{(182, 189)} &
\multicolumn{1}{c}{(544, 565)} & \multicolumn{1}{c}{(3.16, 3.3)} &
\multicolumn{1}{c}{1}\\
AA, $\beta=0.5$, c, $m=8$ & \multicolumn{1}{c}{(207, 214)} &
\multicolumn{1}{c}{(619, 640)} & \multicolumn{1}{c}{(3.57, 3.71)} &
\multicolumn{1}{c}{1}\\\hline
&  &  &  & \vspace{-0.4cm}\\
\multicolumn{5}{p{14cm}}{Note: 99\% conf. interval for the median. 750
parameters. Tolerance set to $\left\Vert f(x)\right\Vert \leq10^{-4}$. 1000
draws.\vspace{0.1cm}}\\\hline
\end{tabular}
\vspace{0.1cm}%
\end{center}%
\end{minipage}%
\end{table}%

Finally, Figure \ref{fig:ancestry} and Table \ref{tab:ancestry} summarize the
results of the numerical experiments with the EM algorithm for admixed
populations. AAmd was again the fastest, although its edge over stationary AA
with $\beta=1$ was small. Composite AAmd was in third place with a clear edge
over both composite AA with constant relaxation.

\subsection{Discussion}

The results shown in Sections \ref{sec:Examples} and \ref{sec:Experiments}
show that adaptive relaxation parameters can clearly reduce the number of
iterations needed for AA to converge. Since maps are often computationally
expensive, a relaxation strategy like AAmd which do not require additional
maps per iteration offers the most predictable benefits. AA with a well-chosen
constant relaxation parameters can approach its performances, but the results
can vary greatly with the choice of $\beta$. A clear advantage of adaptive
relaxation is that they do not require tuning.

Additional numerical experiments (not presented here) were conducted to
explore the impact of each parameter on AAmd. Setting $\beta_{\max}$ higher
than 3 generally made little difference. A smaller $\delta$ (the allowed
discrepancy between $\hat{\beta}_{k}$ and $\hat{\beta}_{k+1}$) made AAmd
slightly slower by falling back on $\beta_{\text{default}}$ more frequently.
Conversly, a larger $\delta$ may not be advisable for highly nonlinear
applications for which optimal relaxation parameters could vary considerably
between iterations. A key parameter is $P$, the number of consecutive
$\hat{\beta}$ values above $1$ before a restart with $\beta=1$. A reasonable
range is $3\leq P\leq10$. Without periodic resets of $\hat{\beta}$, AAmd's
convergence was often slower than that of stationary AA, as seen in Section
\ref{sec:Examples}.

Composite AA was successful in accelerating the estimation of the Bratu
problem but performed poorly for the EM algorithm. Among non-composite AA
implementations, AAopt1 often converged in fewer iterations, though it was not
competitive in terms of computation time. By not recomputing $\beta^{\ast}$ at
each iteration, AAopt1\_4 and AAopt1\_16 were often faster than AA with
constant mappings with $\beta=1$ or $\beta=0.5$, but did not match the speed
of AAmd.

Since AA can be made efficient by using QR decomposition and Givens rotations
to solve the internal linear optimization problem, the fastest algorithms were
clearly those requiring the the fewest mappings to converge. This could change
for applications with truly negligible mapping computation times. However, for
such applications, lighter algorithms that do not require solving any linear
systems, such as SQUAREM \cite{Varadhan2008} or ACX \cite{LepageSaucier2024},
would likely be faster than AA.

\section{Conclusion\label{sec:Conclusion}}

Two adaptive relaxation schemes for Anderson acceleration have been proposed
for convergent fixed-point applications. Both are demonstrated to improve
Anderson acceleration's convergence for a linear contraction mapping.

The first scheme, AAopt1, uses two extra maps to compute a locally optimal
relaxation parameter. Convergence is accelerated by reusing the same maps a
second time to compute the next iterate. Furthermore, by reusing the same
relaxation parameter over multiple iterations, AAopt1\_T always outperforms
AAopt1 in terms of speed, though AA with a constant, well-chosen relaxation
parameter can still be faster.

The second proposed scheme, AAmd, requires fewer iterations to converge while
needing minimal extra calculation. As a result, it outperformed all other AA
specifications across all tests in terms of computation time. Interestingly,
AAmd's adaptive relaxation parameters are frequently above one, an unexpected
result in the context of the AA literature that warrants further investigation.

\bibliographystyle{aaai-named}

\begin{thebibliography}{}

\bibitem[\protect\citename{Alexander \bgroup \em et al.\egroup ,
  }2009]{Alexander2009}
{David H.} Alexander, John Novembre, and Kenneth Lange.
\newblock Fast model-based estimation of ancestry in unrelated individuals.
\newblock {\em Genome Research}, 19:1655--1664, July 2009.

\bibitem[\protect\citename{Anderson, }1965]{Anderson1965}
{Donald G. M.} Anderson.
\newblock Iterative procedures for nonlinear integral equations.
\newblock {\em Journal of the Association for Computing Machinery},
  12(4):547--560, October 1965.

\bibitem[\protect\citename{Anderson, }2019]{Anderson2019}
{Donald G. M.} Anderson.
\newblock Comments on \textquotedblleft {A}nderson acceleration, mixing and
  extrapolation\textquotedblright.
\newblock {\em Numerical Algorithms}, 80:135--234, 2019.

\bibitem[\protect\citename{Bezanson \bgroup \em et al.\egroup ,
  }2017]{bezanson2017}
Jeff Bezanson, Alan Edelman, Stefan Karpinski, and Viral~B. Shah.
\newblock Julia: A fresh approach to numerical computing.
\newblock {\em SIAM Review}, 59(1):65--98, 2017.

\bibitem[\protect\citename{Chen and Vuik, }2022]{Chen2022}
Kewang Chen and Cornelis Vuik.
\newblock Composite {A}nderson acceleration method with two window sizes and
  optimized damping.
\newblock {\em International Journal for Numerical Methods in Engineering},
  123(23):5964--5985, August 2022.

\bibitem[\protect\citename{Chen and Vuik, }2024]{Chen2024}
Kewang Chen and Cornelis Vuik.
\newblock Non-stationary {A}nderson acceleration with optimized damping.
\newblock {\em Journal of Computational and Applied Mathematics}, 451, June
  2024.

\bibitem[\protect\citename{Dempster \bgroup \em et al.\egroup ,
  }1977]{Dempster1977}
Arthur~P. Dempster, Nan~N. Laird, and Donald~B. Rubin.
\newblock Maximum likelihood from incomplete data via the {EM} algorithm.
\newblock {\em Journal of the Royal Statistical Society, Series B (Statistical
  Methodology)}, 39(1):1--38, 1977.

\bibitem[\protect\citename{Dolan and Mor{\'{e}}, }2002]{Dolan2002}
Elizabeth~D. Dolan and Jorge~J. Mor{\'{e}}.
\newblock Benchmarking optimization software with performance profiles.
\newblock {\em Mathematical Programming}, 91:201--213, 2002.

\bibitem[\protect\citename{Evans \bgroup \em et al.\egroup , }2020]{Evans2020}
Claire Evans, Sara Pollock, {Leo G.} Rebholz, and Mengying Xiao.
\newblock A proof that {A}nderson acceleration improves the convergence rate in
  linearly converging fixed-point methods (but not in those converging
  quadratically).
\newblock {\em SIAM Journal on Numerical Analysis}, 58(1):788--810, 2020.

\bibitem[\protect\citename{Fang and Saad, }2009]{Fang2009}
Haw-ren Fang and Yousef Saad.
\newblock Two classes of multisecant methods for nonlinear acceleration.
\newblock {\em Numerical Linear Algebra with Applications}, 16:197--221, 2009.

\bibitem[\protect\citename{Henderson and Varadhan, }2019]{Henderson2019}
{Nicholas C.} Henderson and Ravi Varadhan.
\newblock Damped {A}nderson acceleration with restarts and monotonicity control
  for accelerating {EM} and {EM}-like algorithms.
\newblock {\em Journal of Computational and Graphical Statistics}, May 2019.

\bibitem[\protect\citename{Jin \bgroup \em et al.\egroup , }2024]{Jin2024}
Jiachen Jin, Hongxia Wang, and Kangkang Deng.
\newblock Anderson acceleration of derivative-free projection methods for
  constrained monotone nonlinear equations, 2024.

\bibitem[\protect\citename{Lepage-Saucier, }2024]{LepageSaucier2024}
Nicolas Lepage-Saucier.
\newblock Alternating cyclic vector extrapolation technique for accelerating
  nonlinear optimization algorithms and fixed-point mapping applications.
\newblock {\em Journal of Computational and Applied Mathematics}, 439, March
  2024.

\bibitem[\protect\citename{Pollock and Rebholz, }2021]{Pollock2021}
Sara Pollock and {Leo G.} Rebholz.
\newblock Anderson acceleration for contractive and noncontractive operators.
\newblock {\em IMA Journal of Numerical Analysis}, 41:2841--2872, January 2021.

\bibitem[\protect\citename{Pollock and Rebholz, }2023]{Pollock2023}
Sara Pollock and {Leo G.} Rebholz.
\newblock Filtering for {A}nderson acceleration.
\newblock {\em SIAM Journal on Scientific Computing}, 45(4):A1571--A1590, 2023.

\bibitem[\protect\citename{Potra and Engler, }2013]{Potra2013}
{Florian A.} Potra and Hans Engler.
\newblock A characterization of the behavior of the {A}nderson acceleration on
  linear problems.
\newblock {\em Linear Algebra and its Applications}, 438:1002--1011, November
  2013.

\bibitem[\protect\citename{Pratapa and Suryanarayana, }2015]{Pratapa2015}
{Phanisri P.} Pratapa and Phanish Suryanarayana.
\newblock Restarted {P}ulay mixing for efficient and robust acceleration of
  fixed-point iterations.
\newblock {\em Chemical Physics Letters}, 635:69--74, 2015.

\bibitem[\protect\citename{Pulay, }1980]{Pulay1980}
Peter Pulay.
\newblock Convergence acceleration of iterative sequences. the case of {SCF}
  iteration.
\newblock {\em Chemical Physics Letters}, 73(2):393--398, July 1980.

\bibitem[\protect\citename{Raydan and Svaiter, }2002]{Raydan2002}
Marcos Raydan and {Benar F.} Svaiter.
\newblock Relaxed steepest descent and {C}auchy-{B}arzilai-{B}orwein method.
\newblock {\em Computational Optimization and Applications}, 21:155--167, 2002.

\bibitem[\protect\citename{Saad and Schultz, }1986]{Saad1986}
Youcef Saad and {Martin H.} Schultz.
\newblock {GMRES}: A generalized minimal residual algorithm for solving
  nonsymmetric linear systems.
\newblock {\em SIAM Journal on Scientific and Statistical Computing},
  7(3):856--869, July 1986.

\bibitem[\protect\citename{Tang \bgroup \em et al.\egroup , }2023]{Tang2023}
Bohao Tang, {Nicholas C.} Henderson, and Ravi Varadhan.
\newblock Accelerating fixed-point algorithms in statistics and datascience: A
  state-of-art review.
\newblock {\em Journal of Data Science}, 21(1):1--26, July 2023.

\bibitem[\protect\citename{Varadhan and Roland, }2008]{Varadhan2008}
Ravi Varadhan and Christophe Roland.
\newblock Simple and globally convergent methods for accelerating the
  convergence of any {EM} algorithm.
\newblock {\em Scandinavian Journal of Statistics}, 35:335--353, 2008.

\bibitem[\protect\citename{Walker and Ni, }2011]{Walker2011}
{Homer F.} Walker and Peng Ni.
\newblock Anderson acceleration for fixed-point iterations.
\newblock {\em SIAM Journal on Numerical Analysis}, 49(4):1715--1735, 2011.

\bibitem[\protect\citename{Wang \bgroup \em et al.\egroup , }2016]{Wang2016}
Lianming Wang, {Christopher S.} McMahan, {Michael G.} Hudgens, and {Zaina P.}
  Qureshi.
\newblock A flexible, computationally efficient method for fitting the
  proportional hazards model to interval-censored data.
\newblock {\em Biometrics}, 72:222--231, March 2016.

\bibitem[\protect\citename{Warnock, }2021]{Warnock2021}
Robert Warnock.
\newblock Equilibrium of an arbitrary bunch train with cavity resonators and
  short range wake: Enhanced iterative solution with {A}nderson acceleration.
\newblock {\em Physical Review Accelerators and Beams}, 2021.

\end{thebibliography}

\end{document}